\documentclass[12pt,letterpaper]{amsart}

\usepackage{epsfig,amsmath,amsthm,amssymb,graphicx}
\usepackage[top=1in, bottom=1in, left=1in, right=1in]{geometry}
\usepackage{color}

\usepackage{caption}
\usepackage{subcaption}

\usepackage[dvipsnames]{xcolor}

\usepackage{faktor}

\usepackage{tikz}
\usetikzlibrary{matrix, arrows, cd}
\usepackage{pb-diagram}

\definecolor{myOlive}{rgb}{.29,.28,.16}
\definecolor{myRed}{rgb}{.78,0,0}

\makeatletter
\newtheorem*{rep@theorem}{\rep@title}
\newcommand{\newreptheorem}[2]{%
\newenvironment{rep#1}[1]{%
 \def\rep@title{#2 \ref{##1}}%
 \begin{rep@theorem}}%
 {\end{rep@theorem}}}
\makeatother
	
\newtheorem{proposition}{Proposition}[section]
\newtheorem{theorem}[proposition]{Theorem}
\newtheorem{corollary}[proposition]{Corollary}
\newtheorem{lemma}[proposition]{Lemma}

\newtheorem*{theorem*}{Theorem}
\newtheorem*{proposition*}{Proposition}
\newtheorem*{lemma*}{Lemma}
\newtheorem*{corollary*}{Corollary}

\theoremstyle{definition}
\newtheorem{definition}[proposition]{Definition}
\newtheorem{question}[proposition]{Question}
\newtheorem{conjecture}[proposition]{Conjecture}

\newreptheorem{theorem}{Theorem}
\newreptheorem{proposition}{Proposition}

\theoremstyle{remark}
\newtheorem{remark}[proposition]{Remark}

\newcommand{\bdry}{\partial}

\newcommand{\smooth}{\text{sm}}

\newcommand{\N}{\mathbb{N}}
\newcommand{\Z}{\mathbb{Z}}
\newcommand{\F}{\mathcal{F}}
\newcommand{\G}{\mathcal{G}}
\newcommand{\C}{\mathcal{C}}
\newcommand{\W}{\mathcal{W}}

\newcommand{\R}{\mathbb{R}}

\newcommand{\lk}{\operatorname{lk}}

\renewcommand{\int}[1]{\operatorname{int}(#1)}

\newcommand{\B}{\mathcal{B}}

\newcommand{\onto}{\twoheadrightarrow}
\newcommand{\into}{\hookrightarrow}

\begin{document}
\title[Whitney towers and knots in homology spheres]{Whitney tower concordance and knots in homology spheres}

\author{Christopher W.\ Davis}
\address{Department of Mathematics, University of Wisconsin--Eau Claire}
\email{daviscw@uwec.edu}
\urladdr{www.uwec.edu/daviscw}
\date{\today}


\begin{abstract}
In \cite{Levine2016} A.~Levine proved the surprising result that there exist knots in homology spheres which are not smoothly concordant to any knot in $S^3$, even if one allows for concordances in  homology cobordisms.  Since then subsequent works due to Hom-Levine-Lidman and Zhou \cite{HLL2018, Zhou21} have strengthened this result showing that there are many knots in homology spheres which are not smoothly concordant to knots in $S^3$.  In this paper we  present evidence that the opposite is true topologically.  We  study the Whitney tower filtration of concordance \cite{COT03} and prove that modulo any term in this filtration every knot (or link) in a homology sphere is equivalent to a knot (or link) in $S^3$.  As an application we recover the main result of \cite{Davis2019}, namely that the solvable filtration similarly fails to distinguish links in homology spheres from links in $S^3$.
\end{abstract}

\maketitle
\section{Introduction and statement of results}

%

Two links $K$ and $J$ in $S^3$ are called  \emph{(topologically) concordant} if $K\times\{1\}$ and $J\times\{0\}$ cobound a disjoint union of locally flat embedded annuli in $S^3\times[0,1]$.   
A link is   \emph{(topologically) slice} if it bounds a disjoint union of locally flat embedded disks, called slice disks, in the 4-ball, $B^4$. 
 A link is slice if and only if it is concordant to the unlink.   
 The set of  concordance classes of links is denoted $\C$.  If these embedded annuli and disks are smooth, then we call the resulting links smoothly concordant and smoothly slice respectively.  We denote by $\C_\smooth$ the set of links up to smooth concordance.  The difference between $\C$ and $\C_\smooth$ has been central to the interplay between topological and smooth 4-manifolds.  For instance, the existence of links which are topologically slice but not smoothly slice can be used to produce smooth 4-manifolds which are homeomorphic but not diffeomorphic to $\R^4$ \cite{Gompf85}.  In this paper we present evidence for the following conjectured difference between smooth and topological concordance of links in homology spheres.
 
 \begin{conjecture}\label{conj: main}
 Let $K\subseteq Y$ be a link a a homology sphere.  There is a link $J$ in $S^3$ and a homology cobordism $W$ from $Y$ to $S^3$ in which $K$ and $J$ cobound a disjoint union of locally flat embedded annuli.  Moreover, $W$ can be chosen to be simply connected.
 \end{conjecture}

This conjecture is false in the smooth setting, as proved in  \cite{Levine2016} with subsequent improvements in \cite{HLL2018,Zhou21}.    Recall that a homology cobordism between two 3-manifolds $X$ and $Y$ is informally a 4-manifold bounded by $X\sqcup -Y$ which has the same homology as $X\times[0,1]$.  See also for example, \cite[Chapter 1]{GonzalezThesis}.  
 
This conjecture fits into two extensions of concordance to the setting of links in homology spheres.  The first is $\Z$-concordance.  Two links $K$ and $J$ in homology spheres $X$ and $Y$ are called \emph{
$\Z$-concordant} if there is a homology cobordism from $X$ to $Y$ in which the components of $K$ and $J$ cobound disjoint locally flat embedded annuli.  The $\Z$-concordance class of the unlink is the set of links whose components bound disjoint locally flat embedded disks in a homology ball.  $\widehat{\C}_\Z$ is the set of links in homology spheres up to
 $\Z$-concordance.  

Suppose that $W$ is the the homology cobordism of the previous paragraph.  If $\pi_1(X)\to \pi_1(W)$ and $\pi_1(Y)\to\pi_1(W)$ are each surjective then $W$ is called a \emph{strong homology cobordism} and $K$ and $J$ are called \emph{strongly 
 concordant}.  $\widehat{\C}$ denotes the set of all links up to strong topological concordance.  The strong concordance class of the unlink is the the set of all links in homology spheres which bound disjoint unions of embedded disks in contractible 4-manifolds.    There is an inclusion and a projection map
$$\C\overset{i}{\into} \widehat{\C}\overset{p}{\onto}\widehat{\C}_\Z$$ The injectivity of $i$ follows from the 4-dimensional topological Poincar\'e conjecture \cite{Freedman82}.    Conjecture~\ref{conj: main} can now be interpreted as saying that $\C\to \widehat{\C}$ is a surjection, and as a consequence is bijective.

In \cite{Davis2019} the author presents evidence for Conjecture~\ref{conj: main} in the language of the solvable filtration of Cochran-Orr-Teichner \cite{COT03}.  While this filtration has since received intense study and has been used to detect highly subtle structure in the study of link concordance, it is  the weakest of three filtrations studied in that paper.  The others are the symmetric Whitney tower filtration and the grope filtration.  We  focus on the Whitney tower filtration,
$$
\C \supseteq \W_2\supseteq\W_3\supseteq \dots.
$$

A link in $S^3$ sits in $\W_h$ if it bounds a height $h-1$ symmetric Whitney tower in the 4-ball.  While we give a brief discussion presently and a complete definition in Section~\ref{sect: Whitney towers}, one can think of a Whitney tower as an immersed disk which comes close to admitting a regular homotopy to an embedded disk.  Thus, if a link is in $\W_h$ then it is close to being slice.  See for example \cite[Section 2.1]{CST12} and \cite[Definition 7.7]{COT03}.   A \emph{Whitney tower}, $T$, is a 2-complex in a 4-manifold given by starting with an immersed disk, pairing up its double points with (possibly non-embedded) Whitney disks, as in Figure~\ref{fig: Whit disk}, pairing up the double points amongst these Whitney disks with Whitney disks as in Figure~\ref{fig: Whit tower}, and so on.  If any of these Whitney disks is embedded and has interior disjoint from all other surfaces in $T$, then the Whitney trick of Figure~\ref{fig: Whit move} can be used to remove these two intersection points.  See for example \cite[Section 1.2]{DETBook} for a discussion of the Whitney trick.    If each round of Whitney disks added has interior disjoint from all of the previously added Whitney disks then $T$ is called \emph{symmetric} and the height of $T$ records how many rounds of Whitney disks are added.   For any $h\in \N$, $\W_h$ is the set of all links in homology spheres bounding immersed disks in contractible 4-manifolds that extend to height $h-1$ symmetric Whitney towers.

\begin{figure}[h]
     \centering\begin{subfigure}[b]{0.2\textwidth}
         \centering
         \begin{tikzpicture}
        \node at (0,.2){\includegraphics[width=.5\textwidth]{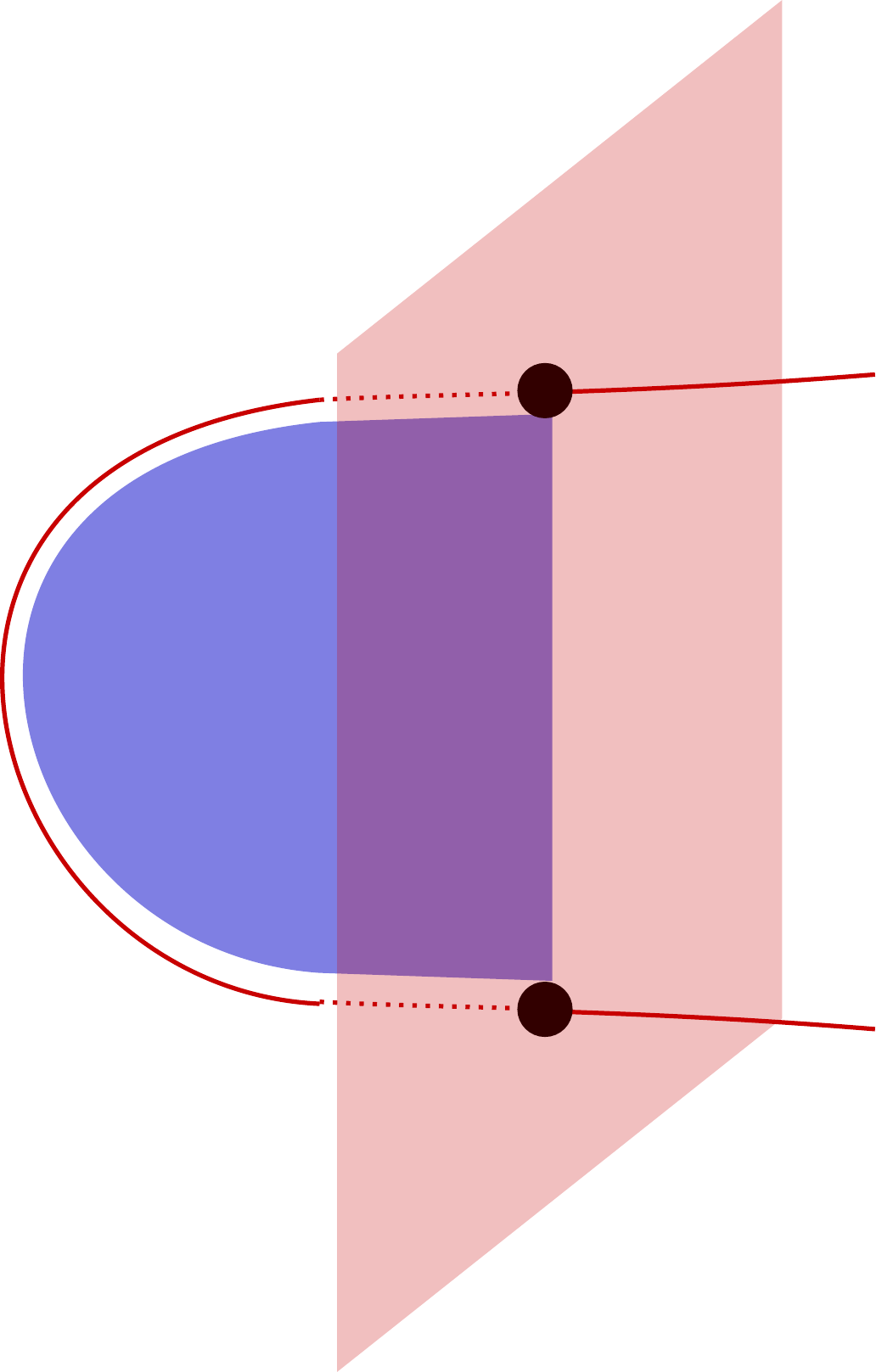}};
         \end{tikzpicture}
     \caption{
     }
     \label{fig: Whit disk}
     \end{subfigure}
     \begin{subfigure}[b]{0.2\textwidth}
         \centering
         \begin{tikzpicture}
        \node at (0,.2){\includegraphics[width=.5\textwidth]{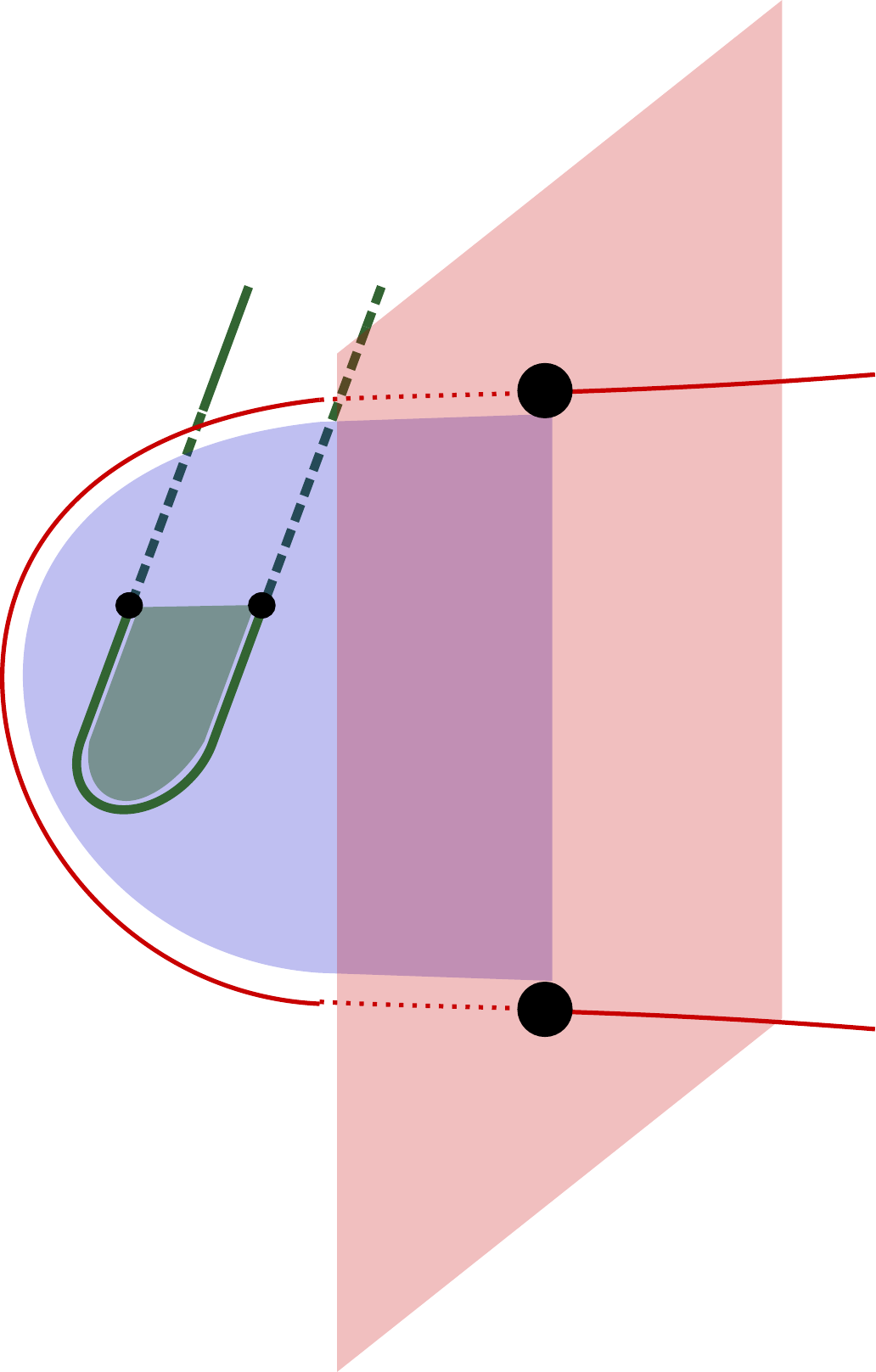}};
         \end{tikzpicture}
     \caption{
     }
     \label{fig: Whit tower}
     \end{subfigure}
     \begin{subfigure}[b]{0.2\textwidth}
         \centering
         \begin{tikzpicture}
        \node at (0,.2){\includegraphics[width=.5\textwidth]{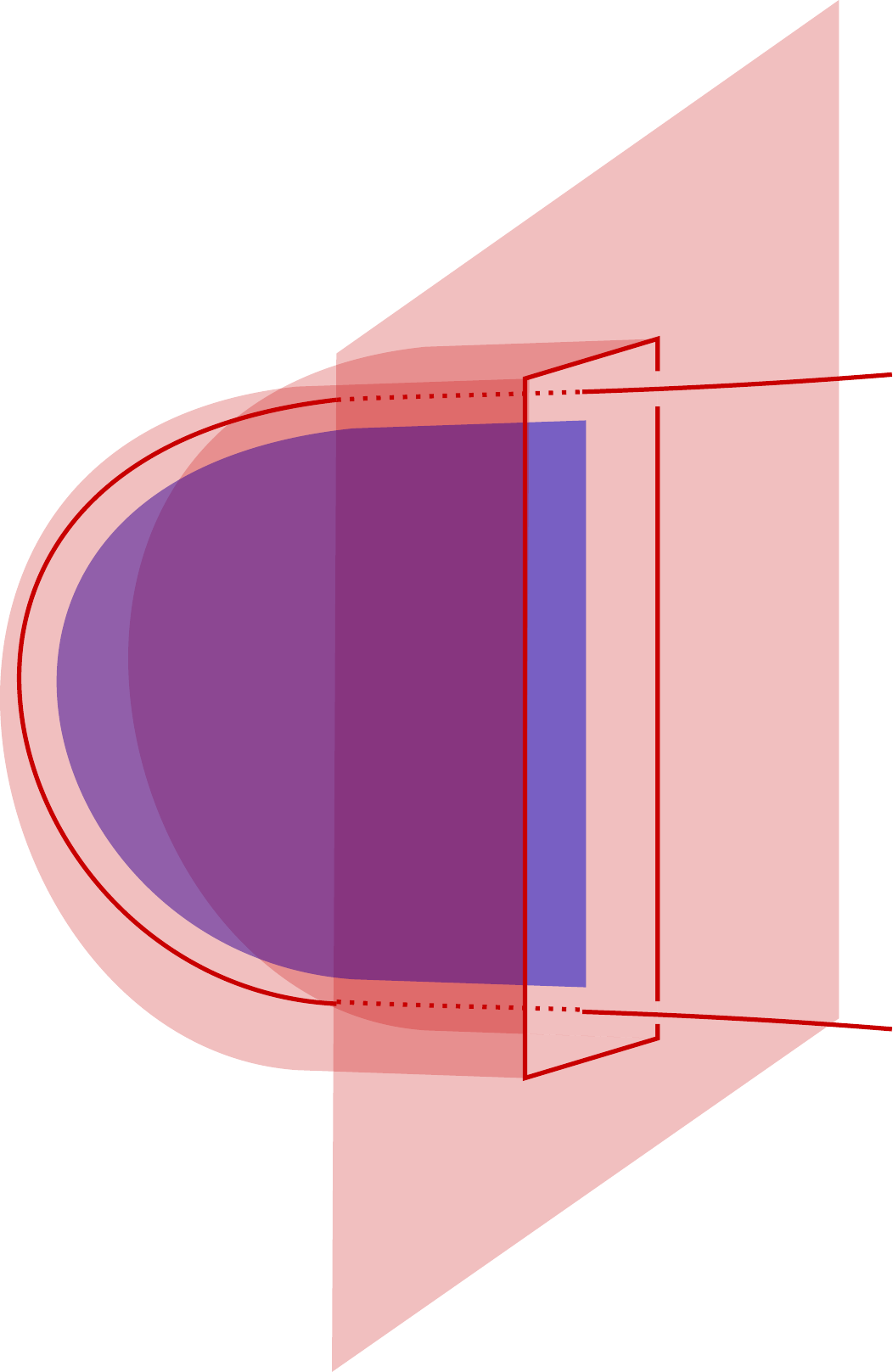}};
         \end{tikzpicture}
     \caption{
     }
     \label{fig: Whit move}
     \end{subfigure}
     
        \caption{Left to right:  A Whitney disk pairing up two points of intersection.  A (portion of a) Whitney tower.  The Whitney moves canceling two points of intersection.
         }
        \label{fig: whit disk and move}
\end{figure}


In \cite[Definition 2.12]{Cha2014}, Cha extends $\W_h$ to an  equivalence relation.  We construct a similar equivalence relation for links in homology spheres as follows:  Links in homology spheres are height $h$ Whitney tower concordant if they cobound an immersed union of annuli in a strong homology cobordism that extends to a height $h-1$ Whitney tower.  See Definition~\ref{defn: symmetric Whitney tower} for a more complete discussion.
Our main result is that every link in a homology sphere is height $h$ Whitney tower concordant to a link in $S^3$.


\begin{theorem}\label{thm:main}
Let $h\in \N$ and $K$ be a link in a homology sphere.  There is some link in $S^3$ which is height $h$ Whitney tower concordant to $K$.  
\end{theorem}

This theorem should be interpreted as saying that any link in a homology sphere cobounds with a link in $S^3$ an immersed union of annuli which is close to admitting a sequence of Whitney moves to a disjoint union of locally flat embedded annuli.  Thus, any link in any homology sphere is close (from the perspective of Whitney towers) to being strongly concordant to a link in $S^3$.  

As mentioned earlier, there are two related filtrations of knot concordance, the solvable filtration $\C\supseteq \F_0\supseteq \F_1\supseteq \dots$ and the grope filtration $\C\supseteq \G_2\supseteq \G_3\dots$ \cite[Definitions 7.9 and 8.5]{COT03}.  The symmetric Whitney tower filtration gives stronger approximations to sliceness than the solvable filtration in that $\W_{h+2}\subseteq \F_h$ \cite[Theorem 8.12]{COT03}.    In \cite{Davis2019} the author defines a sequence of equivalence relations called $n$-solvable concordance and proves that modulo $n$-solvable concordance every link in a homology sphere is equivalent to a link in $S^3$ \cite[Theorem 1.2]{Davis2019}.   In Section~\ref{sect: tower implies sol} we  show  that if links are height $h+2$ Whitney tower concordant then they are $h$-solvably concordant.

\begin{repproposition}{prop: tower to solution}
Let $K$ and $J$ be links in homology spheres.  If $K$ and $J$ are height $h+2$ Whitney tower concordant, then $K$ and $J$ are $h$-solvably concordant.  
\end{repproposition}

We recover the main result of \cite{Davis2019} as a consequence.

\begin{corollary}[Theorem 1.2 of  \cite{Davis2019}]\label{cor:sol}
Let $n\in \N$ and $K$ be a link in a homology sphere.  There is some link $J$ in $S^3$ which is $n$-solvably concordant to $K$.  
\end{corollary}

\begin{question}
The grope filtration $\G_h$ is at least as strong as each of the Whitney tower and solvable filtrations \cite[Corollary 2]{Schniederman06}.  Very informally, two links are height $h$  grope concordant if they cobound a disjoint union of embedded surfaces which is  close to admitting compressing disks. 
See \cite[Definition 2.14]{Cha2014} for the complete definition.   Does this filtration also fail to detect a difference between knots in homology spheres and knots in $S^3$?  More precisely, for any $h\in \N$ and any link $K$ in a homology sphere, is there is a link $J$ in $S^3$ such that the components of $K$ and $J$ cobound a symmetric annular grope of height $h$ in a strong homology cobordism?  
\end{question}

%

\subsection*{Outline of paper.}
 
  In Section~\ref{sect: Whitney towers} we give the formal definition of symmetric Whitney tower concordance and  its relationship with the symmetric Whitney tower filtration of \cite{COT03}.  In order to the  Whitney tower of Theorem~\ref{thm:main}, we  pass through an object called a relative Whitney tower.  These are introduced in \cite{DOP22} where they are used to transform an immersed union of disks bounded by a link $K$ into a sequence of crossing changes from $K$ to a new link bounding a disjoint union of immersed disks.  In Section~\ref{sect: symm rel} we define symmetric relative Whitney towers and explain how to construct symmetric Whitney towers from them.  In Section~\ref{sect: proof of main} we construct symmetric relative Whitney towers and complete the proof of Theorem~\ref{thm:main}.  In Section~\ref{sect: tower implies sol} we recall the notion of solvable concordance and 
   prove that Whitney tower concordance implies solvable concordance.  
  
 \subsection*{Acknowledgements.} I would like to thank Taylor Martin and Carolyn Otto for a thorough reading of a preliminary draft of this manuscript and Mark Powell for suggesting the content of Section~\ref{sect: tower implies sol}.
  
\section{Symmetric Whitney towers.}\label{sect: Whitney towers}

In order to make sense of  Whitney disks and Whitney towers  in a topological 4-manifold one must have notions like immersion, transversality, and tubular neighborhood. A careful exposition on these topics appears in \cite[Section 3]{PowRayTei}.  We give a very brief summary here.  A continuous function $f:S\to W$ from a surface to a 4-manifold is called a \emph{generic immersion} if it locally is a {smooth generic immersion}, meaning that there is an open cover $\{U_\alpha\}_{\alpha\in A}$ of $W$ with homeomorphisms  $\phi_{\alpha}:U_\alpha\to\R^4$ satisfying that for for all $\alpha$, the composition $\phi_{\alpha}\circ f|_{f^{-1}[U_\alpha]}:{f^{-1}[U_\alpha]} \to \R^4$ is a smooth generic immersion.  Generic immersions admit linear normal bundles and so have tubular neighborhoods homeomorphic to a self plumbing of the total space of a vector bundle over $S$ \cite[Section 9.3]{FQ}.  From here on all immersions are assumed to be generic and all immersed surfaces are assumed to be transverse to each other and to the boundary of $W$.  

  Suppose that $p,q\in f[S]$ are double points of $f$ with opposite sign.  Then $f^{-1}\{p\} = \{p_1, p_2\}$ and $f^{-1}\{q\} = \{q_1, q_2\}$.  For $i=1,2$, let $\alpha_i$  be a smoothly embedded arc in $S$ running from $p_i$ to $q_i$.  Assume that other than its endpoints, $f(\alpha_i)$ is disjoint from all double points of $f$ and that $\alpha_1$ and $\alpha_2$ are disjoint.  Then $\alpha:=f(\alpha_1)*f(\overline{\alpha_2})$ is a simple closed curve in $W$.  (Here $*$ indicates concatenation and $\overline{\alpha_2}$ is the  reverse of $\alpha_2$.) Assuming that $\alpha$ is nullhomotopic in $W$, there is an immersed disk $\Delta\subseteq W$ transverse to $S$ bounded by $\alpha$.  

Being an immersed disk, $\Delta$ has a trivial normal  bundle, $N$, admitting a unique trivialization $\psi:N\to D^2\times \R^2$.  In order for $\Delta$ to be called a Whitney disk, the restriction $\psi|_{\alpha}$ of $\psi$ to $\alpha=\bdry \Delta$ must agree with a trivialization coming from $f:S\to W$.  See \cite[Section 2]{Cha2014}.  In order to be explicit, we construct this  trivialization.  Let $U_1, U_2\subseteq S$ be disjoint disk neighborhoods of $\alpha_1$ and $\alpha_2$ respectively.  By taking $U_1$ and $U_2$ to be small enough, we  arrange that $f|_{U_i}:U_i\to W$ is an embedding for $i=1,2$ and that $f[U_1]\cap f[U_2] = \{p_1,p_2\}$.
At any point $x\in f(\alpha_i)$ let $\vec u_i$ be the tangent to $f(U_i)$ normal to $\alpha_i$ and $\vec v_i$ be normal to both $\Delta$ and $f(U_i)$.  Since the points of intersection $p$ and $q$ have opposite signs, we may arrange that at $p$ and $q$,  $\vec u_1= \vec v_2$ and $\vec u_2= \vec v_1$.  These vectors can be chosen to vary continuously and so give another trivialization $\tau$ of $N|_{\alpha}$.  If $\tau=\psi|_{\alpha}$ then $\Delta$ is a \emph{Whitney disk} for $S$ pairing $p$ and $q$. 

\begin{definition}[Definition 7.7 of \cite{COT03}, Definition 2.5 of \cite{Cha2014}]\label{defn:Symm WT}
Let $S$ be a framed immersed surface in a 4-manifold.  A \emph{symmetric Whitney tower} $T$ of height $h\in \N$ based on $S$ is a sequence $T_0, T_1,\dots, T_h$ of immersed surfaces such that 
\begin{itemize}
\item $T_0=S$, 
\item for $j\ge 1$, $T_j$ is a collection of transverse framed immersed Whitney disks with disjoint boundaries pairing up all of the double points of $T_{j-1}$, and 
\item for $j>i$, $T_j$ has interior disjoint from $T_i$.
\end{itemize}  
We  call $T_i$ the \emph{$i$'th stage of $T$}, and Whitney disks in $T_i$ are called \emph{height $i$}.  If such a $T$ exists then we say that $S$ \emph{extends to $T$}.
\end{definition}

There is a more general notion simply called  a Whitney tower.  For a survey the reader is directed to \cite{Schneiderman20}.  All Whitney towers in this paper are  symmetric.  As a consequence any time we refer a Whitney tower $T$ it should be understood that $T$ is symmetric, even if we omit the word.

\begin{remark}\label{rmk: off by one}
There is a discrepancy between the notion of height in \cite{COT03} and \cite{Cha2014}.  In \cite[Definition 7.7]{COT03} the  immersed surface $S$ is called height 1, while in \cite[Definition 2.5]{Cha2014} it is height $0$.  The definition above follows the conventions of \cite[Definition 2.5]{Cha2014}.  
\end{remark}

\begin{remark}\label{rmk: smoothing}
As is explored in \cite[Remark 2.19]{Cha2014}, we could do equally well  writing the definition above in the smooth setting.  Indeed, suppose that $T\subseteq W$ is a Whitney tower in a 4-manifold.  Let $p\in W\setminus T$.  By  work of Freedman-Quinn \cite[Theorems 8.2]{FQ}, $W\setminus\{p\}$ admits a smooth structure.  As $T$ is immersed in $W$, it has a tubular neighborhood $\nu(T)$ homeomorphic to a self plumbing of a vector bundle over $T$.  Thus, $T$ is smooth with respect to a smooth structure on $\nu(T)$, which admits another smooth structure as a submanifold of $W\setminus \{p\}$.  By \cite[Theorems 8.1A]{FQ}, after changing $T\subset \nu(T)$ by a smooth regular homotopy and an ambient isotopy of $\nu(T)$, $T$ is smooth with respect to the smooth structure coming from $W\setminus \{p\}$.  Since regular homotopy consists of the Whitney move along embedded Whitney disks and its inverse, this operation does not change the height of the resulting Whitney tower.  
\end{remark}

\begin{definition}\label{defn: COT}
Let $K$ be a link in a homology sphere $Y$ and $\B$ be a contractible 4-manifold bounded by $Y$.  We say that $K\in \W_h$ if there is a framed immersed union of disks in $\B$ bounded by $K$, with framing extending the $0$-framing on $K$ and which extends to a height $h-1$ Whitney tower.  
\end{definition} 

This definition inspires an equivalence relation on the set of links in homology spheres.  Instead of building a Whitney tower starting with a union of immersed disks, start with immersed annuli.  

\begin{definition}[Compare with Definition 2.12 of \cite{Cha2014}.]\label{defn: symmetric Whitney tower}
Let $h\in \N$,   $K$ and $J$ be $n$-component links in homology spheres $X$ and $Y$, and $W$ be a strong homology cobordism from $X$ to $Y$.   
For $i=1,\dots, n$ let $A_i$ be an immersed annulus with $\bdry A_i = K_i\cup -J_i$ and over which the $0$-framing on $K_i\cup-J_i$ extends.
  If $A=A_1\cup\dots\cup A_n$ extends to a height $h-1$ symmetric Whitney tower $T$, then we say that $K$ is \emph{height $h$ Whitney tower concordant} to $J$ and write $K\sim_h J$.  We call $T$ a \emph{height $h$ Whitney tower concordance}.
\end{definition}

In the above definition $-J$ indicates the reverse of $J$ sitting in the orientation reverse of $Y$.  We close this section by proving that the definitions of $\W_h$ and $\sim_h$ are compatible, in the sense that a link is in $\W_h$ if and only if it is $\sim_h$-equivalent to the unlink.  


\begin{proposition}\label{prop: compatible}
Let $K$ be a link in a homology sphere and $h\in \N$.  Then $K$ is height $h$ Whitney tower concordant to the unlink in $S^3$ if and only if $K\in \W_h$.
\end{proposition}
\begin{proof}
Suppose that $K\subseteq Y$ is a link in a homology sphere and $U\subseteq S^3$ be the unlink.  Let $W$ be a strong homology cobordism from $Y$ to $S^3$ and $A\subseteq W$ be an immersed union of annuli bounded by $K$ and $U$.  Furthermore, assume that the $0$ framing on $K\cup U$ extends over $A$ and $A$ extends to a height $h-1$ symmetric Whitney tower, $T$.  A schematic appears in Figure~\ref{fig: WT conc}.  Let $\B$ be the contractible 4-manifold resulting from capping the $S^3$-boundary component of $W$ with a 4-ball.  It follows from the fact that $W$ is a strong homology cobordism that $\B$ is a simply connected homology ball, and so is contractible.  Let $D\subseteq B^4$ be a disjoint union of smoothly embedded disks bounded by $U$.  Notice that $A$ and $D$ both induce the $0$-framing on $U$ so that $A\cup D \subseteq \B$ is an immersed union of disks in a contractible 4-manifold over which the $0$-framing of $K$ extends and which forms the base surface of a height $h-1$ symmetric Whitney tower, $T\cup D$.  Thus, $K\in \W_h$.  A schematic of $A\cup D\subseteq \B$ appears in Figure~\ref{fig:construct slice disk}.

\begin{figure}[h]
     \centering\begin{subfigure}[b]{0.3\textwidth}
         \centering
         \begin{tikzpicture}
        \node at (0,.5){\includegraphics[width=.8\textwidth]{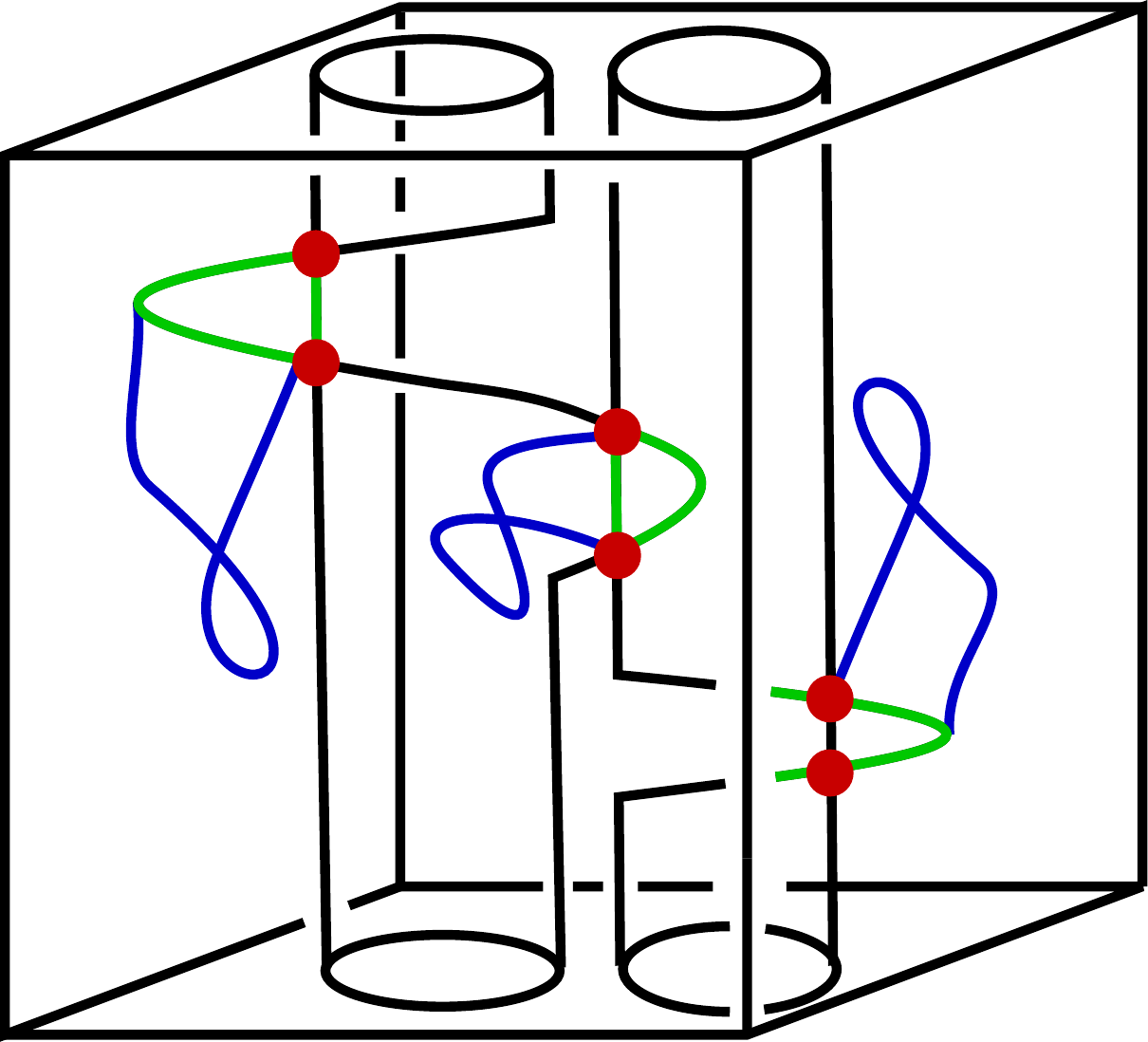}};
        \node at (0,0){\phantom{\includegraphics[width=.8\textwidth]{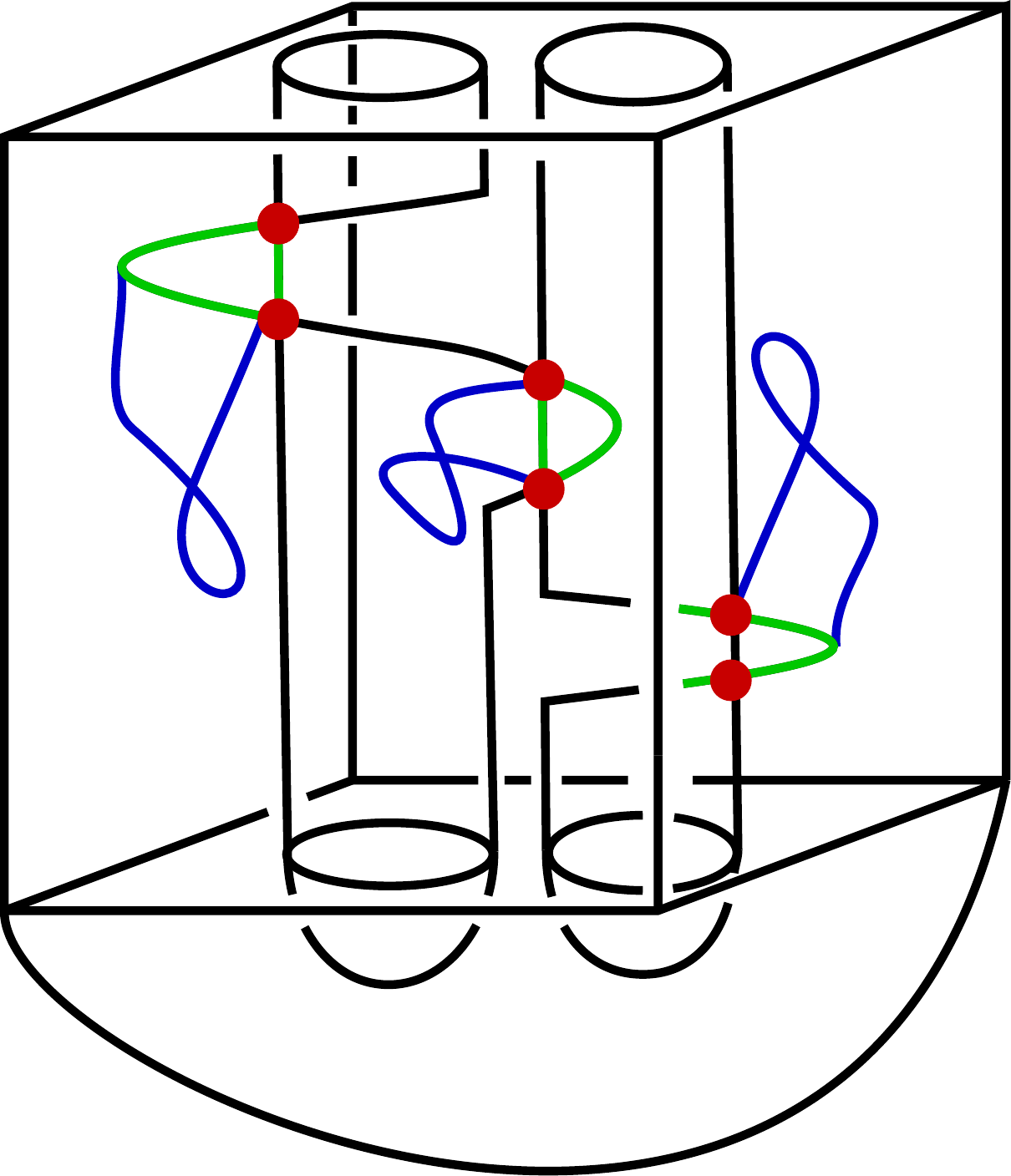}}};
        \node at (-1.6,2.4) {$Y$};
        \node at (1.6,-1.2) {$S^3$};
        \node at (0,3.0) {$K$};
        \draw[<-](.3,2.3)--(0,2.8);
        \draw[<-](-.3,2.3)--(0,2.8);
        \node at (0,-2.0) {$U$};
        \draw[<-](.3,-1.3)--(0,-1.8);
        \draw[<-](-.3,-1.3)--(0,-1.8);
         \end{tikzpicture}
     \caption{
     }
     \label{fig: WT conc}
     \end{subfigure}
     \begin{subfigure}[b]{0.3\textwidth}
         \centering
         \begin{tikzpicture}
        \node at (-1.6,2.4) {$Y$};
        \node at (0,3.0) {$K$};
        \draw[<-](.3,2.3)--(0,2.8);
        \draw[<-](-.3,2.3)--(0,2.8);

        \node at (0,0){\includegraphics[width=.8\textwidth]{ConcordanceToSlice2.pdf}};
         \end{tikzpicture}
     \caption{
     }\label{fig:construct slice disk}

     \end{subfigure}
     \begin{subfigure}[b]{0.3\textwidth}
         \begin{tikzpicture}
        \node at (0,0){\includegraphics[width=.8\textwidth]{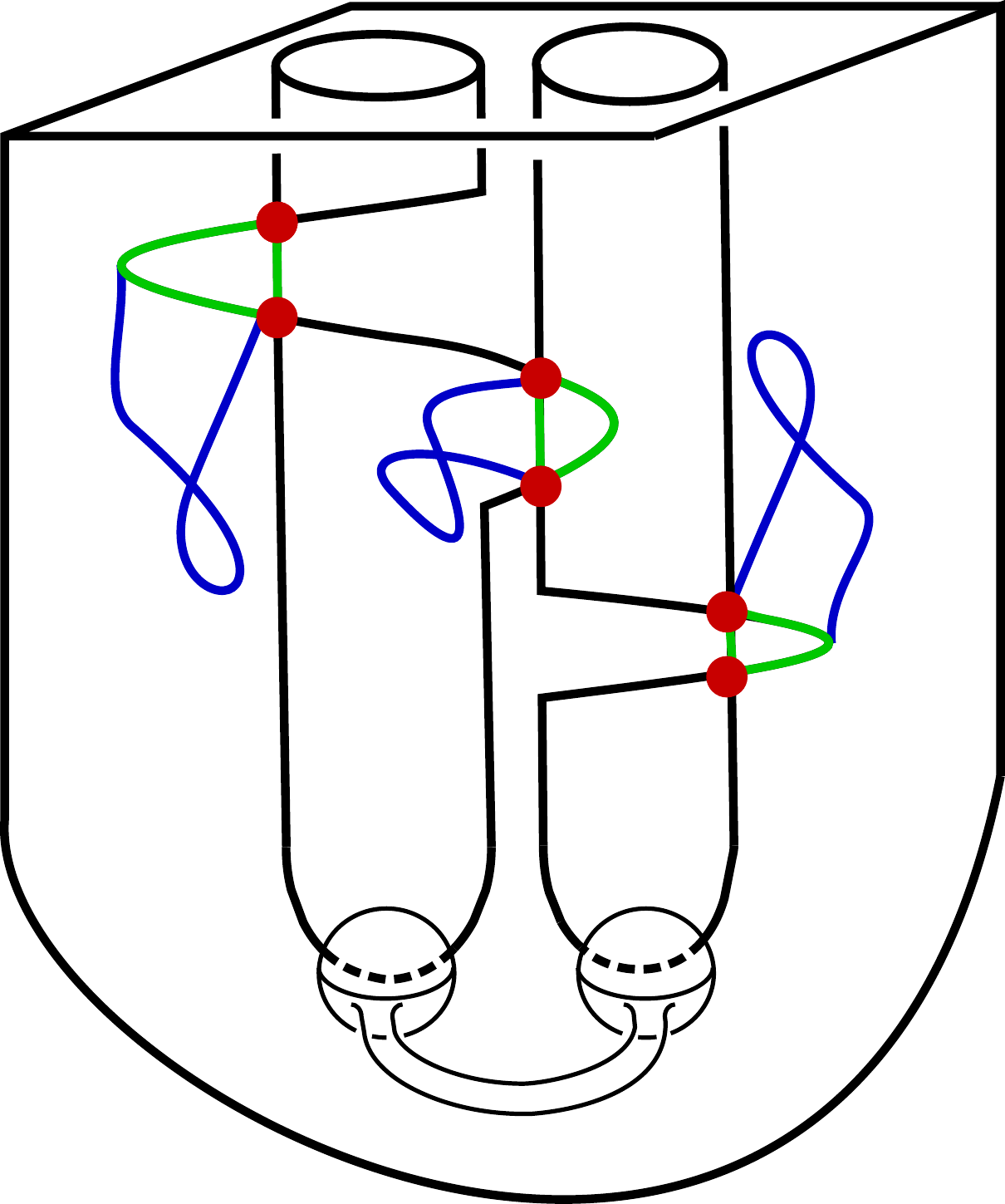}};
        
        \node at (-1.6,2.4) {$Y$};
        \node at (0,3.0) {$K$};
        \draw[<-](.3,2.4)--(0,2.8);
        \draw[<-](-.3,2.4)--(0,2.8);
        \node at (-1.6, -.2) {$D_1$};
        \draw[->](-1.3,-.2)--(-.9,-.2);
        \node at (1.6, 1.5) {$D_2$};
        \draw[->](1.4,1.5)--(1,1.5 );
        
        \node at (-1.4, -1) {$B_1$};
        \draw[->](-1.2,-1)--(-.8,-1.4);
        
        \node at (1.5, -1) {$B_2$};
        \draw[->](1.3,-1)--(.9,-1.4);

         \end{tikzpicture}
         \caption{
     }\label{fig: slice to concordance}
     \end{subfigure}
     
        \caption{Left: An immersed union of annuli $A$ bounded by $K\subseteq Y$ and $U\subseteq S^3$ extending to a height $1$ Whitney tower.  Middle: Capping $A$ with disks bounded by $U$ produces an immersed  union of disks extending to a height $1$ Whitney tower.  Right:  A immersed union of disks $D_1\cup D_2$ bounded by a link $K$ together with a pair of small 4-balls $B_1$ and $B_2$.  Tube $B_1$ and $B_2$ together to get a single 4-ball $B$.  Remove the interior of $B$ to arrive at a Whitney tower concordance from $K$ to the unlink.
         }
     \label{fig: concordance to slice}
\end{figure} 

Conversely, suppose that $\B$ is a contractible 4-manifold bounded by $Y$, that $D\subseteq \B$ is a framed immersed union of disks bounded by the 0-framing of $K$ and that $D$ extends to a height $h-1$ symmetric Whitney tower, $T$ in $\B$.  
For any $i$, let $K_i$ be the $i$'th component of $K$ and $D_i\subseteq D$ be the immersed disk bounded by $K_i$.  Pick a point $p_i$ interior to $D_i$, disjoint from all of the Whitney disks in $T$ and a small closed 4-ball neighborhood $B_i\subseteq \mathcal{B}$  of $p_i$.
  By taking $B_i$ small enough, we arrange that $\bdry B_i\cap D_i$ is an unknot and that $B_i\cap D_i = B_i\cap T$ is an embedded disk bounded by this unknot. 
   For $i=1,\dots, n-1$, let $\beta_i$ be an embedded arc in $\B$ running from $B_i$ to $B_{i+1}$.  A dimensionality argument allows us to arrange that $\beta_i$ is disjoint from $T$.  Let $B\subseteq \mathcal{B}$ be the 4-ball given by tubing together $B_1,\dots, B_n$ along $\beta_1,\dots, \beta_{n-1}$.  See Figure~\ref{fig: slice to concordance} for a schematic.

%
%
%

We now see a strong homology cobordism $\B\setminus \int{B}$ from $Y$ to $S^3$, in which $A=D\setminus \int{B}$ is an immersed union of annuli bounded by $K$ and $U$.  Here $\int{B}$ indicates the interior of $B$.   That $W$ is a strong homology cobordism follows from the fact that $\B$ is contractible.  We restrict the framing on $D$ to $D\cap B$.  Since $D\cap B$ is a collection of slice disks for $U$, this framing restricts to the $0$-framing on $U$.  Thus, the $0$-framing on $K\cup U$ extends over $A$.   Finally $A$ is the base surface for a height $h-1$ symmetric Whitney tower,  $T\setminus \int{B}$.  Thus, $K$ is height $h$ symmetric Whitney tower concordant to the unlink in $S^3$.
\end{proof}

\section{Symmetric relative Whitney towers}\label{sect: symm rel}

Symmetric Whitney towers are a restriction of a larger class of immersed 2-complexes in 4-manifolds, simply called Whitney towers.  In \cite{DOP22} the author along with Orson and Park define a variation of a Whitney tower called a {relative Whitney tower}.  Relative Whitney towers have two important properties.  First,  they exist much more readily than actual Whitney towers \cite[Lemma 5.4]{DOP22}.   Second,  they can be used to guide a homotopy of an immersed surface which changes the boundary in a controlled way to a new immersed surface which extends to a Whitney tower \cite[Lemma 5.5]{DOP22}.  In order to prove Theorem~\ref{thm:main} we   extend these two results to the setting of symmetric Whitney towers.  

Let $\psi:S\to W$ be an immersion of a surface (possibly with corners) into a 4-manifold.  Assume that $\psi$ is transverse to the boundary so that $\psi^{-1}[\bdry W] \subseteq \bdry S$ is a disjoint union of arcs and circles.  
Let $p$ be a double point in $\psi(S)$ with preimage $\psi^{-1}\{p\} = \{p_1, p_2\}$.  Let $\alpha_1, \alpha_2:[0,1]\into S$ be disjoint embedded arcs in $S$ running from $p_1$ and $p_2$ to points $q_1$ and $q_2$ interior to $\psi^{-1}(\bdry W)$, $\alpha_1$ and $\alpha_2$ interior to $S$ except for the endpoints at $q_1$ and $q_2$.  Let $\alpha_3$ be an arc in $\bdry W$ running from $\psi(q_1)$ to $\psi(q_2)$, which is otherwise disjoint from $\psi[S]$.  The concatenation $\psi(\alpha_1)*\alpha_3*\psi(\overline{\alpha_2})$ is a simple closed curve in $W$.  If this curve bounds an immersed disk $\Delta$ transverse to $S$ then $\Delta$ is called a \emph{relative Whitney disk} associated to the double point $p$.  The arc $\alpha_3$ is called the \emph{relative Whitney arc} of $\Delta$.  A schematic appears in Figure~\ref{fig: relWhitDisk}.

\begin{figure}[h]
     \centering
     \begin{subfigure}[b]{0.3\textwidth}
         \centering
         \begin{tikzpicture}
        \node at (0,.2){\includegraphics[width=.8\textwidth]{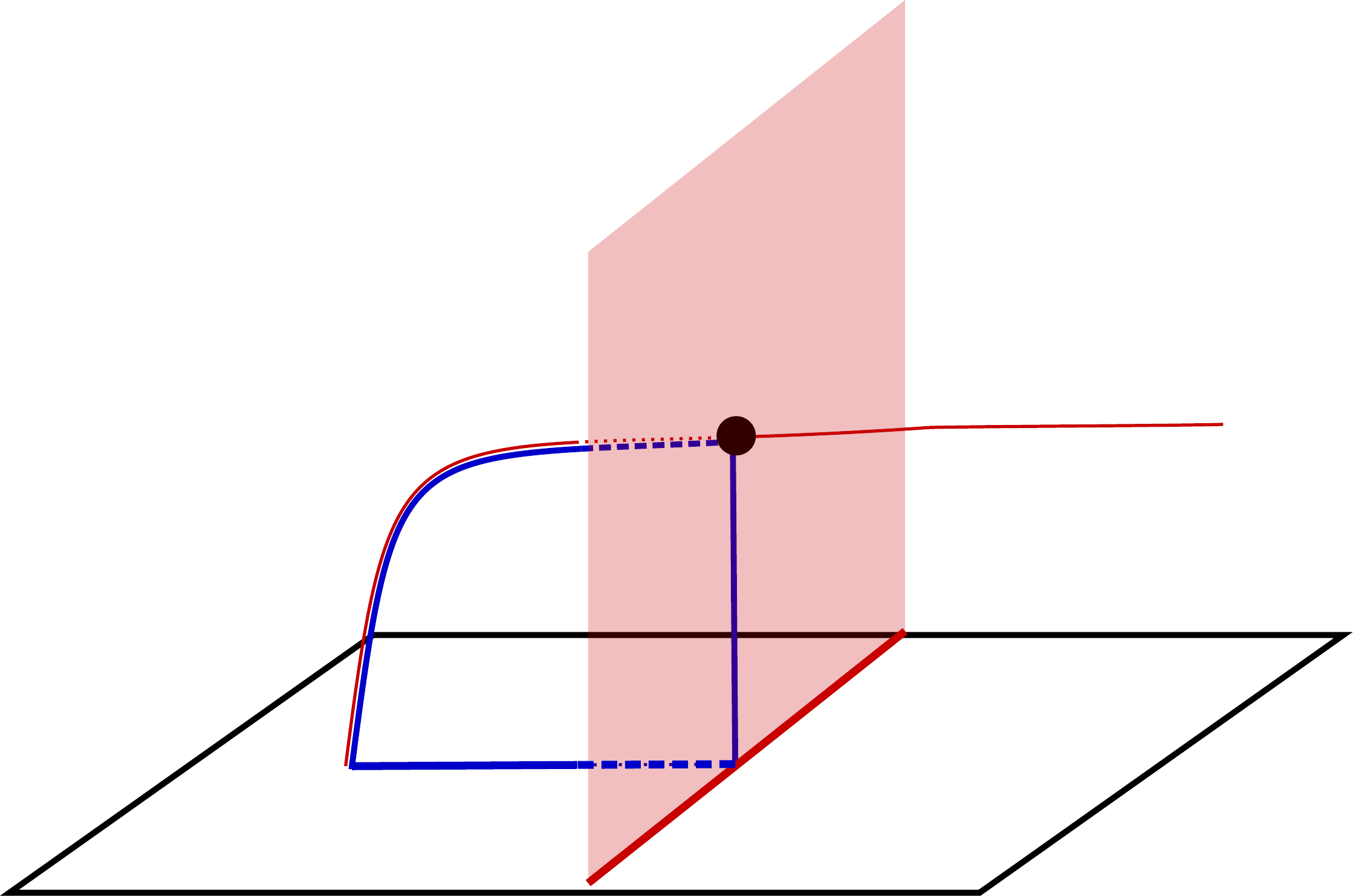}};
        \node at (1,-.6) {\tiny$\bdry W$};
        \node at (.2,.5) {\tiny$p$};
        \node at (-1,-.1) {\tiny$\alpha_1$};
        \node at (.4,-.1) {\tiny$\alpha_2$};
        \node at (-.6,-.9) {\tiny$\alpha_3$};
         \end{tikzpicture}
     \caption{
     }
     \label{fig: rel Whit arc}
     \end{subfigure}
     \begin{subfigure}[b]{0.3\textwidth}
         \begin{tikzpicture}
        \node at (0,.2){\includegraphics[width=.8\textwidth]{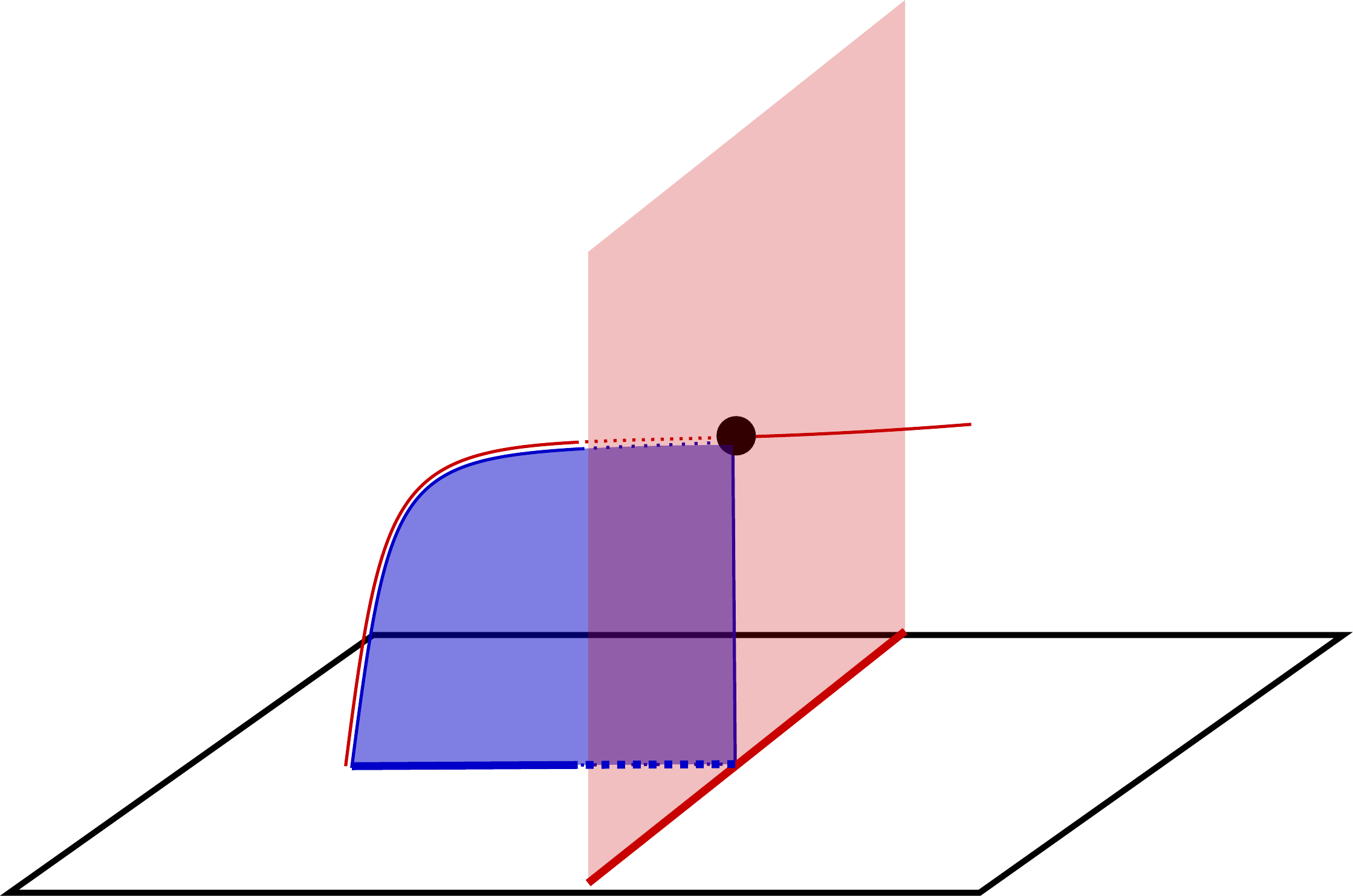}};
        \node at (1,-.6) {\tiny$\bdry W$};
        \node at (-.45,-.1) {\tiny$\Delta$};
         \end{tikzpicture}
     \caption{
      }
     \label{fig: rel Whit disk}
     \end{subfigure}
     \begin{subfigure}[b]{0.3\textwidth}
         \centering
         \begin{tikzpicture}
        \node at (0,.2){\includegraphics[width=.8\textwidth]{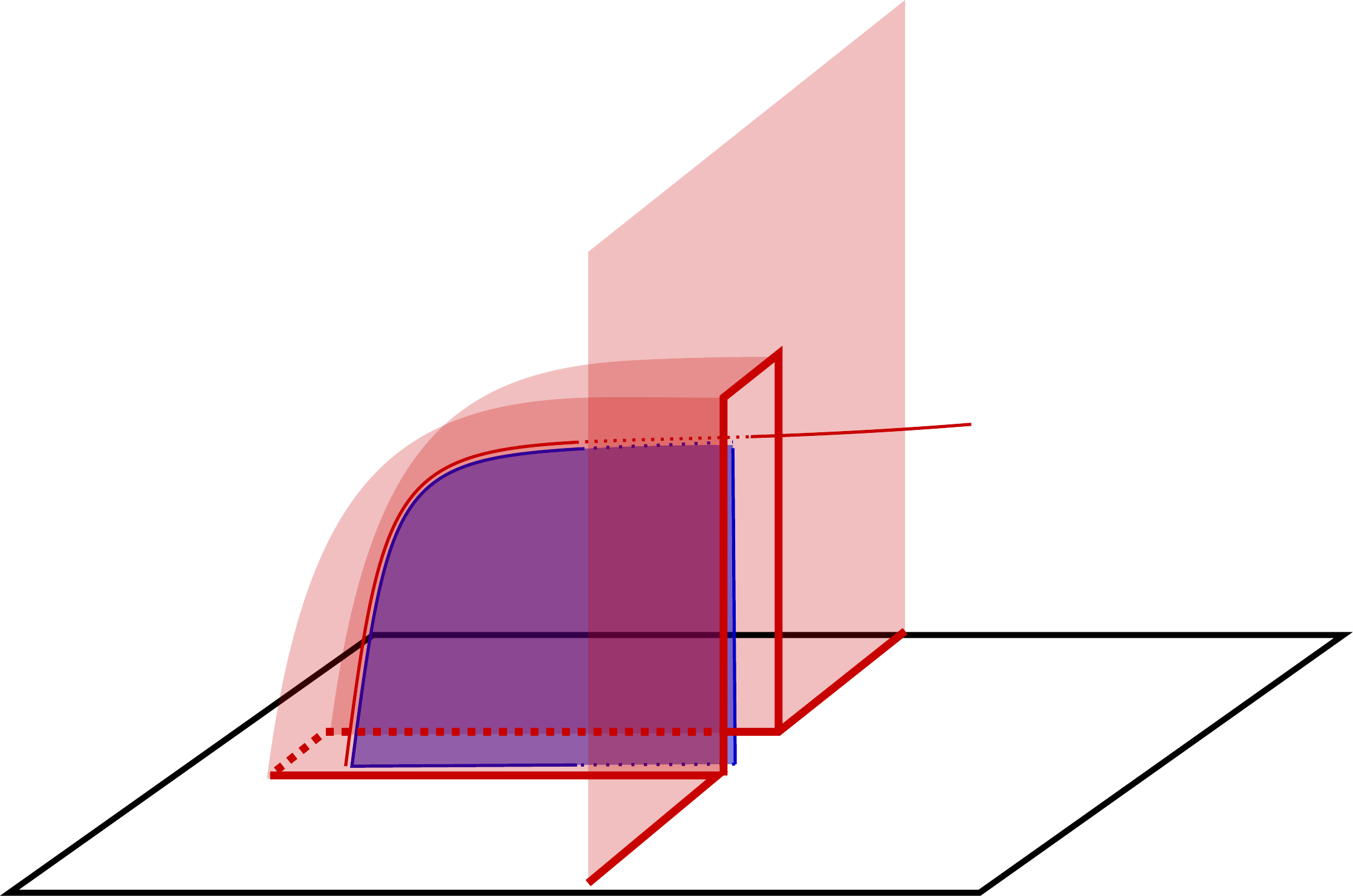}};
        \node at (1,-.6) {\tiny$\bdry W$};
         \end{tikzpicture}
     \caption{}
     \label{fig: rel Whit move}
     \end{subfigure}
        \caption{Left: A double point $p$ in an immersed surface together with arcs $\alpha_1$ and $\alpha_2$ running from $p$ to points in $\bdry W \cap S$ and a third arc $\alpha_3$ connecting the endpoints of $\alpha_1$ and $\alpha_2$.  Center: $\alpha_1*\alpha_3*\overline{\alpha_2}$ bounding an immersed disk, called a relative Whitney disk.  Right: The relative Whitney trick along $\Delta$ removes the double point at $p$.  }
        \label{fig: relWhitDisk}
\end{figure}
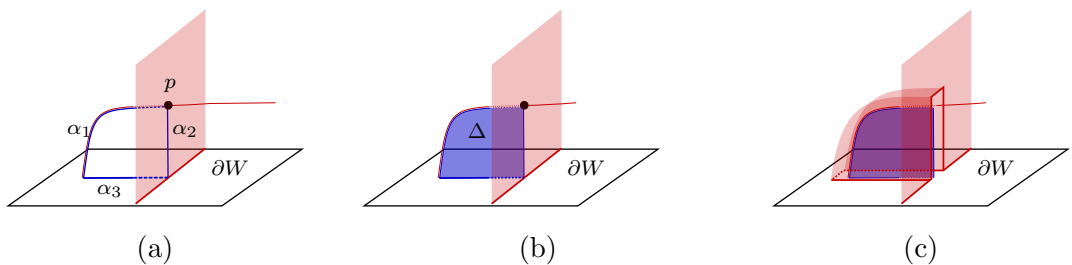 

The reader will notice that while the definition of  a Whitney disk requires some discussion of trivialization of normal bundles, the definition of a relative Whitney disk does not.  This is because a trivialization of a vector bundle on an arc on the boundary of a of a disk automatically extends to a trivialization over the disk.   Thus, if $N$ is the normal bundle of $\Delta$ and $N|_{\psi(\overline{\alpha_2})*\psi(\alpha_1)}$ is its restriction to $\psi(\overline{\alpha_2})*\psi(\alpha_1)$ then any trivialization of $N|_{\psi(\overline{\alpha_2})*\psi(\alpha_1)}$ automatically extends over $N$.   We construct a trivialization of $N|_{\psi(\overline{\alpha_2})*\psi(\alpha_1)}$ as follows.  At any point in $\psi(\alpha_i)$ let $\vec u_i$ be a tangent vector to $\psi(S)$ which is normal to $\psi(\alpha_i)$ and let $\vec v_i$ be normal to $\Delta$ and to $\psi(S)$.  These may be chosen to vary continuously and so that at the double point $p$, $\vec u_1=\vec v_2$ and $\vec v_2=\vec u_1$.  These two normal vectors together give a trivialization of $N|_{\psi(\overline{\alpha_2})*\psi(\alpha_1)}$ which we extend to give a framing of $\Delta$.  Notice that this imposes a framing on the relative Whitney arc, $\alpha_3$.  See \cite[Section 2.2]{DOP22} for a more detailed discussion. 


\begin{definition}
Let $S$ be an immersed surface in a 4-manifold.  A \emph{symmetric relative Whitney tower} of height $h\in \N$ based on $S$ is a sequence $T_0$, $T_1$,\dots, $T_h$ of immersed surfaces such that 
\begin{itemize}
\item $T_0=S$, 
\item for $j\ge 1$, $T_j$ is a collection of transverse relative Whitney disks with disjoint boundaries consisting of one relative Whitney disk associated to each double point in $T_{j-1}$, and 
\item for $j>i$,  $T_j$ has interior disjoint from $T_i$.
\end{itemize}  We  call $T_i$ the $i$'th stage of $T$, and relative Whitney disks in $T_i$ are called height $i$.  

\end{definition}


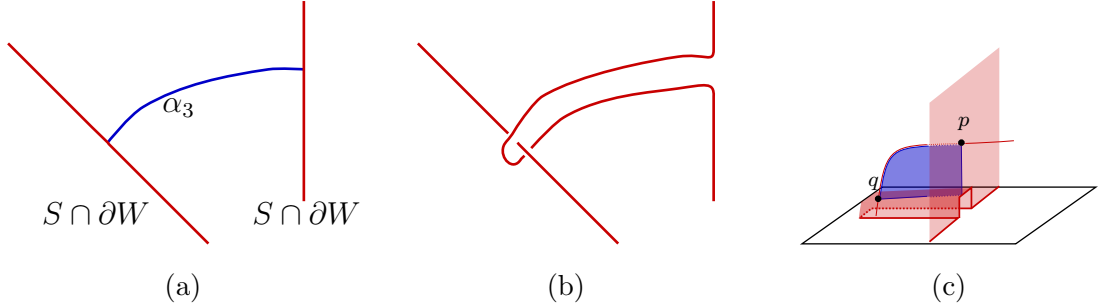
\begin{figure}[h]
     \centering
     \begin{subfigure}[b]{0.3\textwidth}
         \centering
         \begin{tikzpicture}
        \node at (0,.2){\includegraphics[width=.8\textwidth]{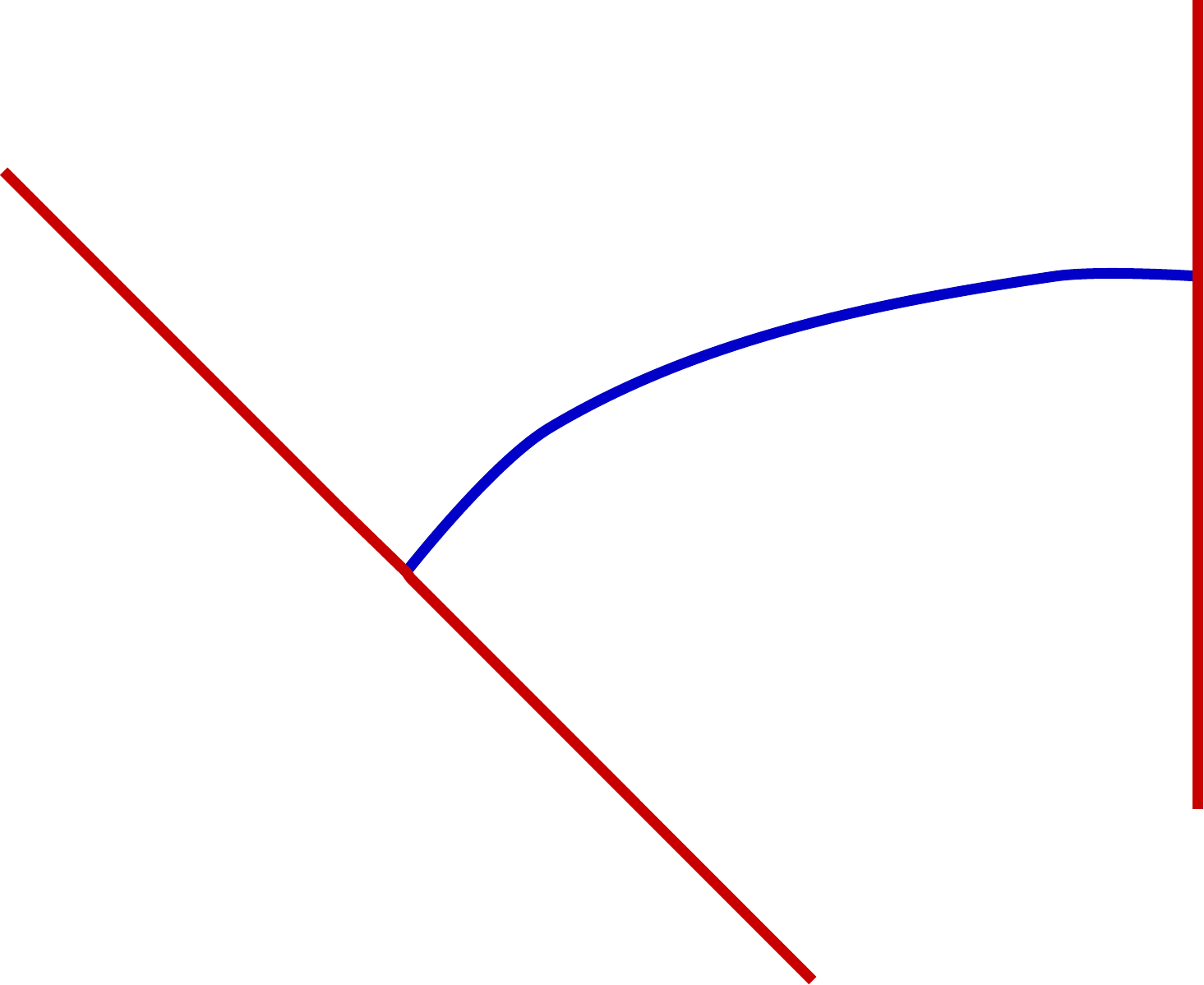}};
        \node at (.3,.4) {$\alpha_3$};
        \node at (-.8,-1) {$S\cap \bdry W$};
        \node at (2,-1) {$S\cap \bdry W$};
         \end{tikzpicture}
     \caption{
     }
     \label{fig: rel Whit arc in bdry}
     \end{subfigure}
     \begin{subfigure}[b]{0.3\textwidth}
         \centering
         \begin{tikzpicture}
        \node at (0,.2){\includegraphics[width=.8\textwidth]{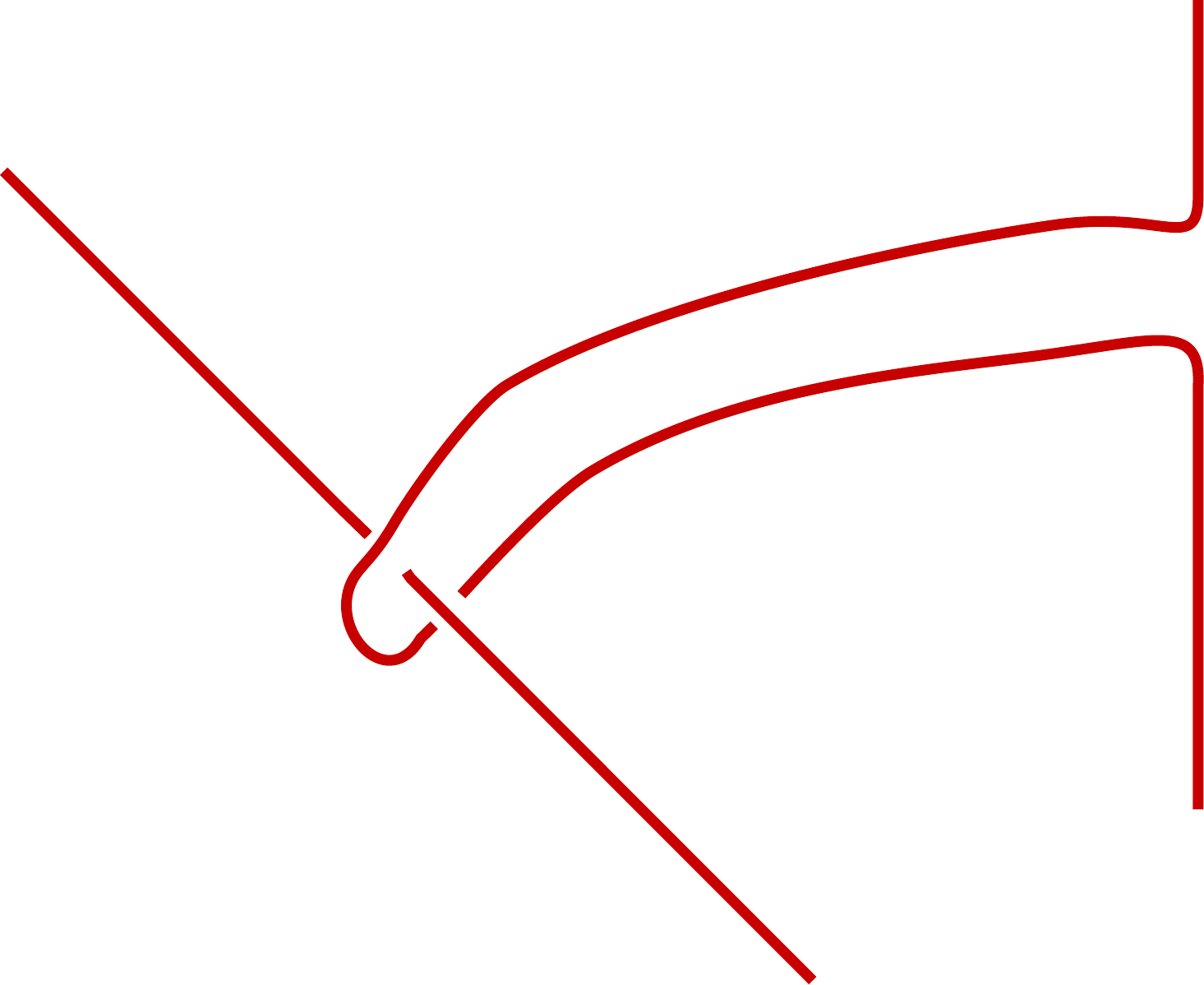}};
         \end{tikzpicture}
     \caption{
     }
     \label{fig: rel Whit move in bdry}
     \end{subfigure}
     \begin{subfigure}[b]{0.3\textwidth}
         \centering
         \begin{tikzpicture}
        \node at (0,.2){\includegraphics[width=.8\textwidth]{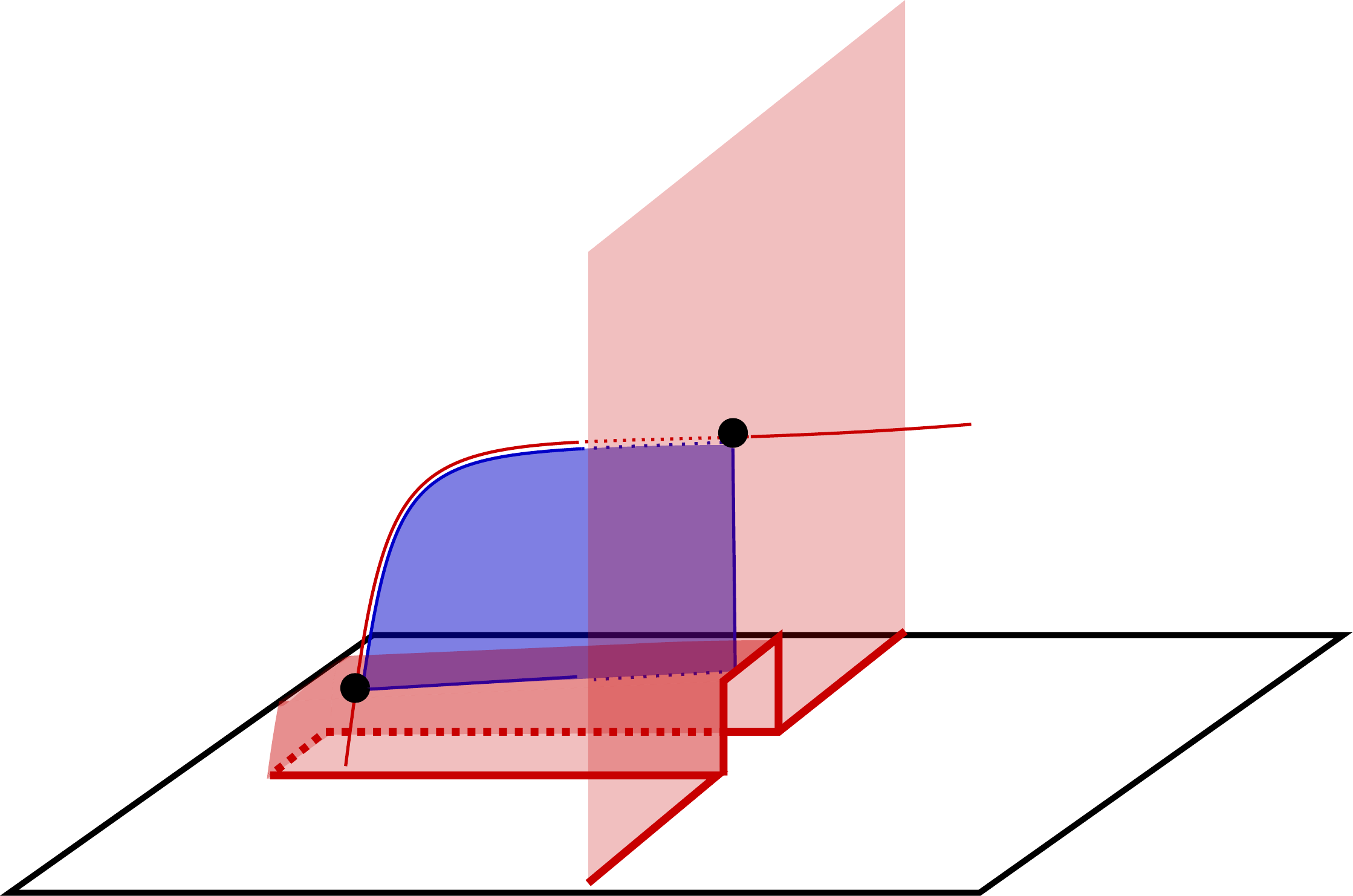}};
        \node at (.2,.5) {\tiny$p$};
        \node at (-1,-.3) {\tiny$q$};
         \end{tikzpicture}
     \caption{
     }
     \label{fig: partial rel Whit move}
     \end{subfigure}
        \caption{Left: Two arcs in $S\cap \bdry W$ connected by the relative Whitney arc $\alpha_3$ of a relative Whitney disk.  Center: The relative Whitney move changes $S\cap \bdry W$ by sliding along $\alpha_3$ over a meridian.  Right: The partial relative Whitney trick introduces a new point of intersection and  replaces a relative Whitney disk with a Whitney disk. }
        \label{fig: rel Whit }
\end{figure} 

An important application of relative Whitney disks is the so-called relative Whitney trick, schematically described in Figure~\ref{fig: rel Whit move} and described in detail in \cite[Subsection 2.2]{DOP22}.  Another is the \emph{partial relative Whitney trick} of Figure~\ref{fig: partial rel Whit move} \cite[Proof of Lemma 5.5]{DOP22}.  We give the properties of the partial relative Whitney trick.  Let $p$ be a double point in an immersed surface $A$. Let $\Delta$ be a relative Whitney disk associated with $p$ and $\alpha\subseteq \bdry W$ be its relative Whitney arc.  The partial relative Whitney trick
\begin{itemize}
\item changes $A$ by a homotopy which is constant away from a 
 small neighborhood of $\alpha$,
\item introduces exactly one new double point $q$ to $A$,
\item replaces $\Delta$ with a Whitney disk pairing the intersection points $p$ and $q$, and 
\item changes $\bdry A$ by a homotopy pushing an arc in $\bdry A$ along the framed relative Whitney arc over a meridian of $\bdry A$, as in Figure~\ref{fig: rel Whit move in bdry}.
\end{itemize}

Let $T$ be a symmetric relative Whitney tower.  By performing the partial relative Whitney move on all of the relative Whitney disks of $T$ one gets a symmetric Whitney tower.

\begin{lemma}[Compare to Lemma 5.5 of \cite{DOP22}]\label{lem: relative Whitney tower to Whitney tower}
Let $W$ be a 4-manifold and $T$ be a height $h$ symmetric relative Whitney tower in $W$.  If $Y$ is a connected submanifold of $\bdry W$ and contains all of the relative Whitney arcs of $T$, then there exists a symmetric relative Whitney tower, $T'$ of height $h$ such that the height 0 surfaces of $T$ and $T'$ differ by a homotopy which is constant outside of a small neighborhood of $Y$. 
\end{lemma}
\begin{proof}
As the proof is identical to that of \cite[Lemma 5.5]{DOP22}, we give only a summary.  
Start with a symmetric relative Whitney tower $T$ of height $h$.  Perform the partial relative Whitney move along each of the height $h$ relative Whitney disks in $T$.  These moves change $T_{h-1}$ by a homotopy to $T_{h-1}'$ (a new collection of relative Whitney disks if $h-1>0$).  We now have $T_0\cup T_1\cup\dots\cup T_{h-2}\cup T_{h-1}'$ is a height $h-1$ relative Whitney tower and $T_{h-1}'$ extends to a height $1$ symmetric Whitney tower.

Now do the same along all height $h-1$ relative Whitney disks, so that $T_0\cup T_1\cup\dots\cup T_{h-3}\cup T_{h-2}'$ is a height $h-2$ relative Whitney tower and $T_{h-2}'$ extends to a height $2$ Whitney tower.  Iterate until you have performed the relative Whitney trick along every relative Whitney disk.  What results is a height $0$ relative Whitney tower (i.e.~ an immersed surface) $T_0'$ which extends to a height $h$ Whitney tower.  Moreover $T_0'$ differs from $T_0$ by a homotopy supported in a small neighborhood of the relative Whitney arcs of $T$.  As $Y$ contains all of the relative Whitney arcs, this completes the proof.
\end{proof}

\section{Constructing symmetric relative Whitney towers: The proof of Theorem~\ref{thm:main}.  }\label{sect: proof of main}

 Let $S$ be an immersed surface and $\Delta$ be a Whitney disk pairing two double points of $S$.  If $\Delta$ is embedded and has interior disjoint from $S$ then the Whitney trick can be used to remove these points of intersection.  The \emph{finger move} is an inverse to the Whitney trick.  Let $S\subseteq W$ be an immersed surface in a 4-manifold.  Let $\beta$ be framed arc running from a point interior to $S$ to another point interior to $S$.  A schematic for the result of changing $S$ by a finger move along $\beta$ is depicted in Figure~\ref{fig: finger move}.  In the result of the finger move there are two new intersection points paired by an embedded Whitney disk.  Performing the Whitney move along this disk results in a surface isotopic to $S$.  See also \cite[Section 11.2]{DETBook}.  
 
 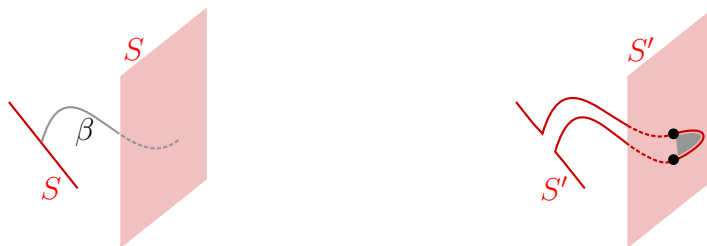
\begin{figure}[h]
     \centering\begin{subfigure}[b]{0.4\textwidth}
         \centering
         \begin{tikzpicture}
        \node at (0,.2){\includegraphics[width=.4\textwidth]{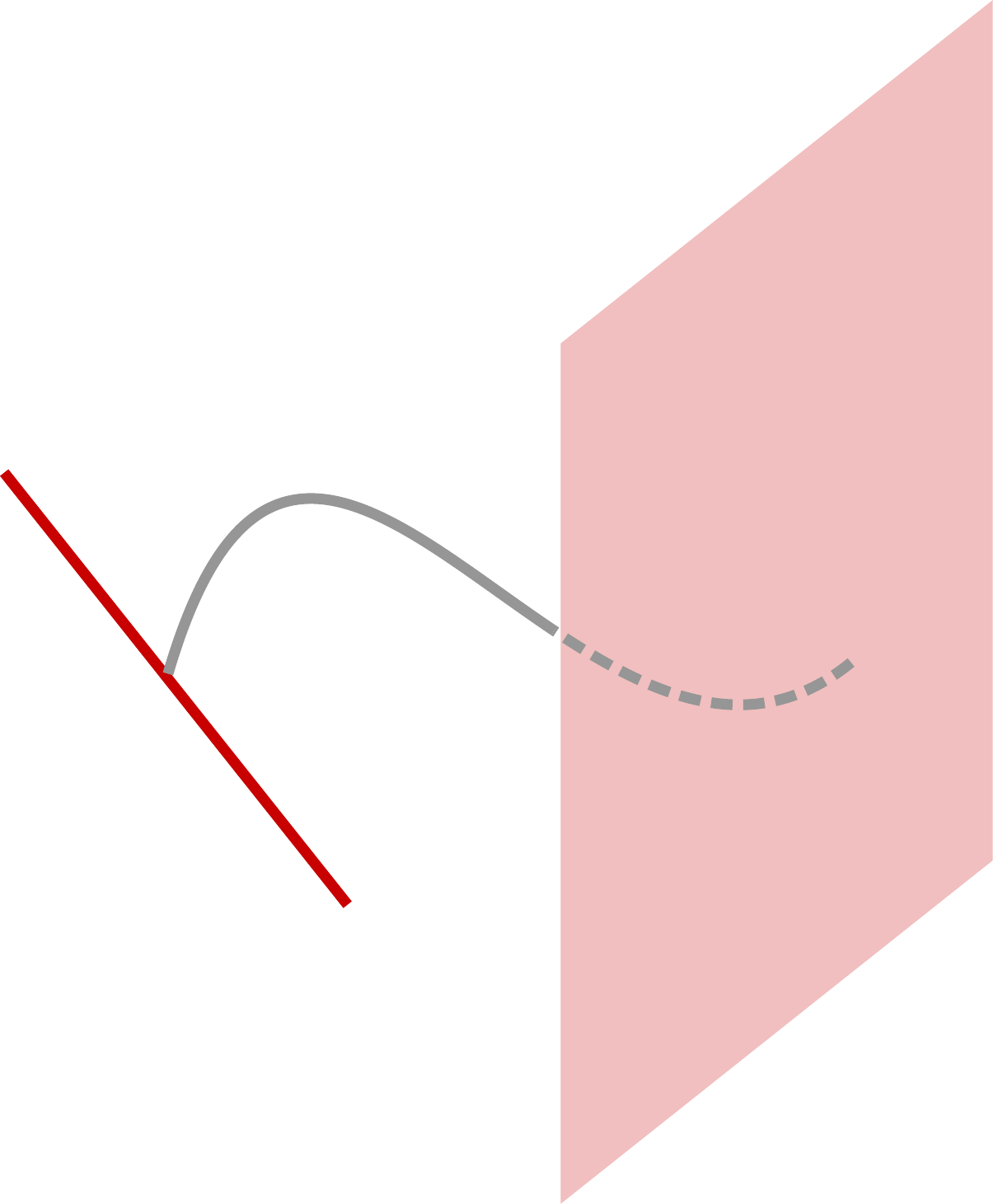}};
        \node[red] at (-.75,-.6) {$S$};
        \node[red] at (.35,1.25) {$S$};
        \node at (-.3,.15) {$\beta$};
         \end{tikzpicture}
     \end{subfigure}
     \begin{subfigure}[b]{0.4\textwidth}
         \centering
         \begin{tikzpicture}
        \node at (0,.2){\includegraphics[width=.4\textwidth]{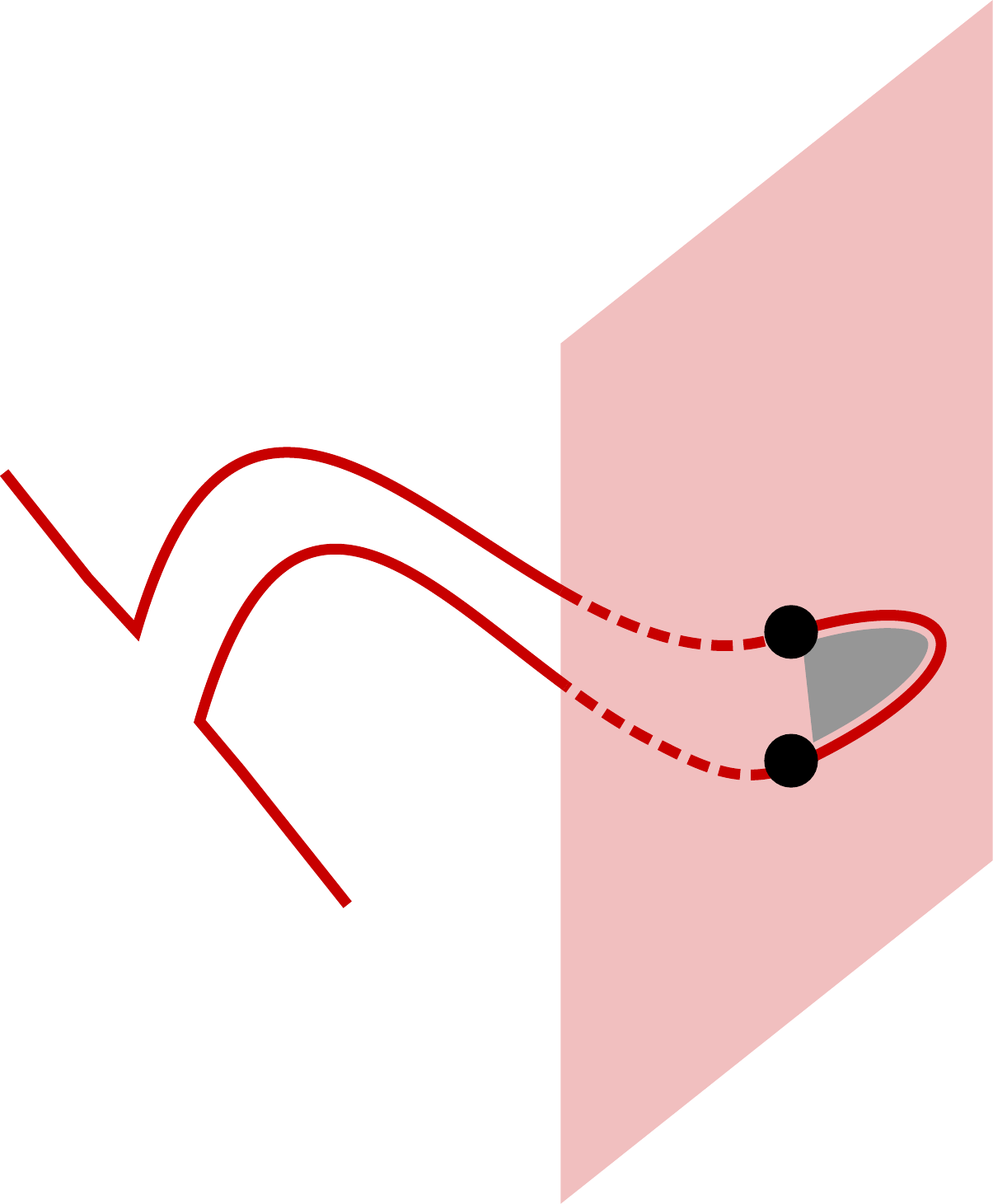}};
        \node[red] at (-.8,-.6) {$S'$};
        \node[red] at (.35,1.25) {$S'$};
         \end{tikzpicture}
     \end{subfigure}
        \caption{Left: An arc $\beta\subseteq W$ between two points in immersed surface $S\subseteq W$.  Right: The result of the finger move $\beta$ along with an embedded Whitney disk which undoes the finger move.  
         }
        \label{fig: finger move}
\end{figure} 

 Lemma~\ref{lem: relative Whitney tower to Whitney tower}
 gives a path to the proof of Theorem~\ref{thm:main}.  Let $K$ be a link in a homology sphere $Y$.  By \cite[Corollary 9.3]{FQ} $Y$ bounds a simply connected homology ball.  By removing a small open 4-ball from its interior one gets a simply connected homology cobordism $W$ from $Y$ to $S^3$.  Since $W$ is simply connected, there is a free homotopy (a union of immersed annuli) in $W$ from $K$ to the unlink $U$ in $S^3$.  Call this immersed union of annuli $A$.  If we can arrange that $A$ extends to a height $h$ symmetric relative Whitney tower, then we will be able to use Lemma~\ref{lem: relative Whitney tower to Whitney tower} to replace it with a Whitney tower concordance from $K$ to some link in $S^3$.   The first step is arranging that $\pi_1(S^3\setminus \nu(U))\onto \pi_1(W\setminus \nu(A)) $ is surjective.

 \begin{lemma}\label{lem: surface pi1 onto}
Let $W$ be a 4-manifold with boundary.  Let $S$ be an immersed surface in $W$.  Let $Y\subseteq \bdry W$ be a connected 3-manifold.  Assume that every  component of $S$ has nonempty intersection with $Y$ and that the inclusion induced map $\pi_1(Y)\onto \pi_1(W)$ is onto.  

There is an immersed surface $S'$ in $W$ satisfying that 
the inclusion induced map $\pi_1(Y\setminus \nu(S'))\onto \pi_1(W\setminus \nu(S'))$ is onto and that $S'$ differs from $S$ by sequence of finger moves.  In particular $\bdry S = \bdry S'$.
 \end{lemma}
 \begin{proof}
 Let $W$, $S$ and $Y$ be as in the lemma.  Pick a base point $y$ in $Y\setminus \nu(S)$.  Let $\{g_1,\dots, g_k\}$ be a generating set for $\pi_1(W\setminus \nu(S), y)$.  Since $\pi_1(Y,y)\onto \pi_1(W,y)$ is surjective, for each $i$, there is some $g_i'\in \pi_1(Y,y)$ and an immersed disk $E_i$ bounded by $g_i*\overline{g_i'}$.
 A dimensionality argument allows us to arrange that $g_i'$ is disjoint from $\bdry S$. If $E_i$ is disjoint from $S$, then $g_i$ is already in the image of $\pi_1(Y\setminus \nu(S))\to \pi_1(W\setminus \nu(S))$.  Otherwise, let $p$ be a transverse point of intersection between $E_i$ and $S$.  Let $\alpha_1$ and $\alpha_2$ be arcs in $S$ and $E_i$ from $p$ to points in $Y$.  Since $\pi_1(Y,y)\onto \pi_1(W,y)$, there is an arc $\alpha_3$ in $Y$ so that $\alpha_1*\alpha_3*\overline{\alpha_2}$ is nullhomotopic and so bounds a relative Whitney disk, $\Delta$ associated with $p$.  
 
 We would like to use the relative Whitney trick to eliminate the point of intersection at $p$, but $\Delta$ might be neither embedded nor disjoint from $S$.  Let $r$ be a transverse point of intersection in $\Delta \cap S$.  Find an arc $\beta$ in $\Delta$ from $r$ to a point interior to  $\alpha_1\subseteq S$.  As in Figure ~\ref{fig: finger move in surface} we  modify $S$ by a finger move along $\beta$ in order to reduce the number of points in $\Delta\cap S$ by 1.  This move adds two new double points to $S$.  As is briefly explained in \cite[Chapter XII \S 2]{Kirby89}, the finger move changes $\pi_1(W-\nu(S))$ by killing a commutator of two meridians of $S$. In particular, $g_1,\dots, g_n$ is still a generating set for $\pi_1(W-\nu(S))$.  Iterating, we  arrange the $\Delta \cap S=\emptyset$.   Finally, we modify $E_i$ by the relative Whitney trick along $\Delta$.  This eliminates the point of intersection $p\in S\cap E_i$.  It may add double points to $E_i$ and will change $g_i'$.  Regardless, we now have that $g_i$ is in the image of $\pi_1(Y\setminus \nu(S)) \to \pi_1(W\setminus \nu(S),y)$.  By following the steps in this paragraph for every $i=1,\dots, n$ we arrange that $\pi_1(Y\setminus \nu(S)) \onto \pi_1(W\setminus \nu(S),y)$ is surjective.  Notice that this process changes $S$ by some finger moves and  does not change $\bdry S$.

\begin{figure}[h]
     \centering\begin{subfigure}[b]{0.4\textwidth}
         \centering
         \begin{tikzpicture}
        \node at (0,.2){\includegraphics[width=.8\textwidth]{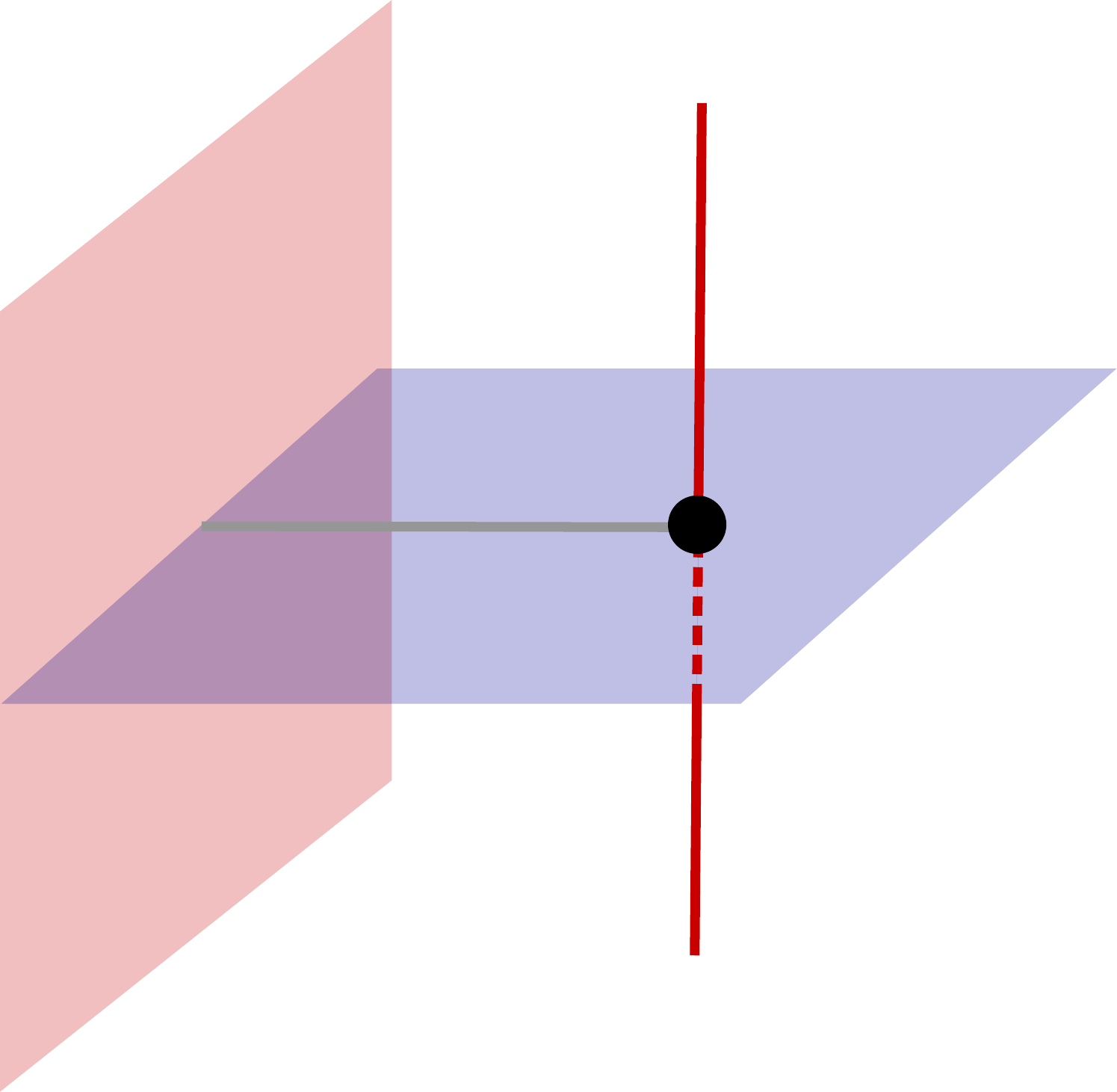}};
        \node[] at (1,.3) {$r$};
        \node[red] at (.9,2.2) {$S$};
        \node[red] at (-2,-1) {$S$};
        \node[blue] at (1.8,.8) {$\Delta$};
        \node at (-.3,.5) {$\beta$};
         \end{tikzpicture}
     \end{subfigure}
     \begin{subfigure}[b]{0.4\textwidth}
         \centering
         \begin{tikzpicture}
        \node at (0,.2){\includegraphics[width=.8\textwidth]{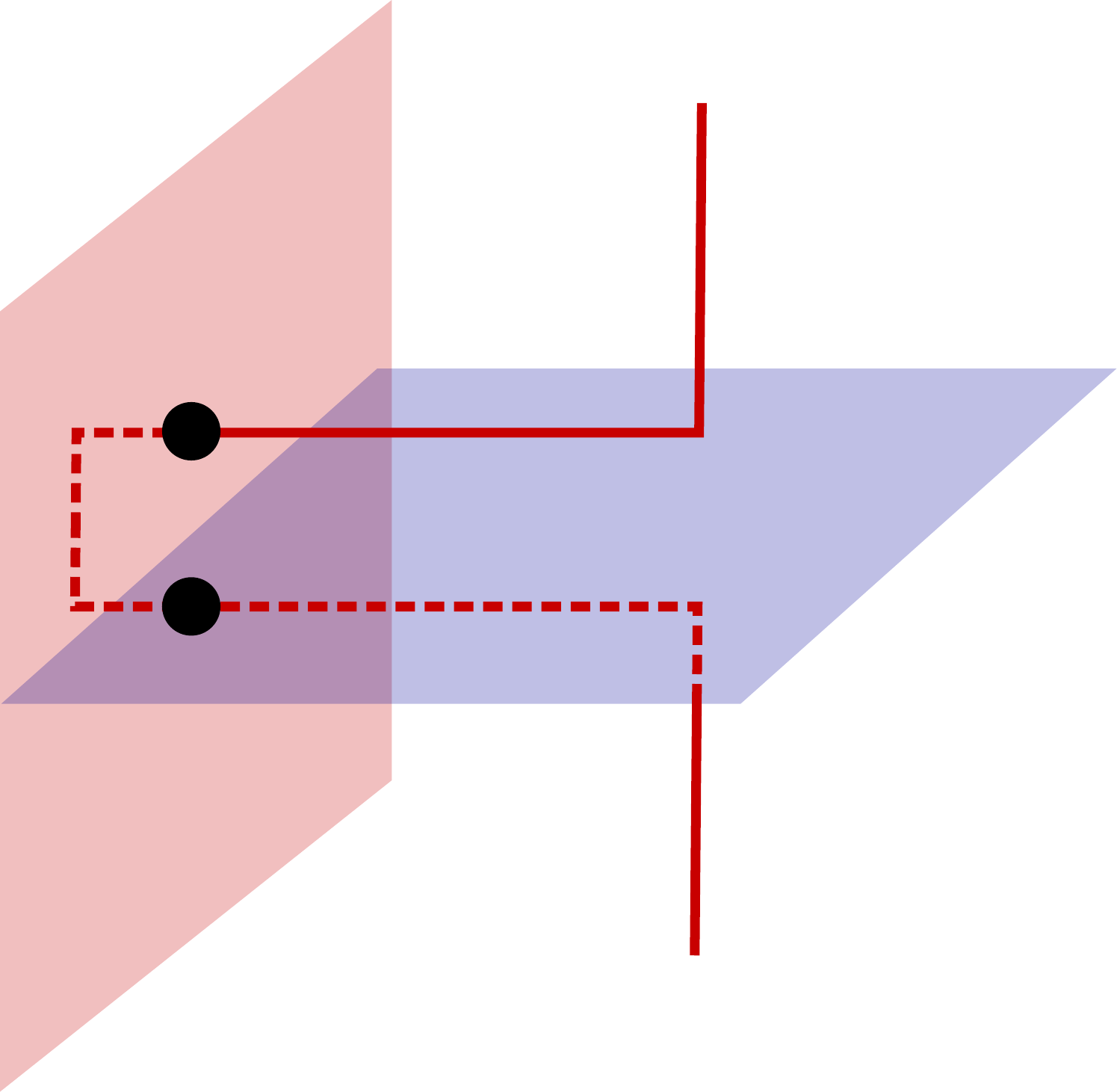}};
        \node[red] at (.9,2.2) {$S$};
        \node[red] at (-2,-1) {$S$};
        \node[blue] at (1.8,.8) {$\Delta$};
         \end{tikzpicture}
     \end{subfigure}
        \caption{Left: A point $q$ in the intersection of a surface $S$ with a Whitney disk $\Delta $ associated to a point in $E_i\cap S$ together with an arc $\beta\subseteq \Delta$ from $r$ to a point in $S\cap \bdry \Delta$  Right: A finger move reduces the number of  points in $S\cap \Delta$ at a cost of two new double points in $S$.  
         }
        \label{fig: finger move in surface}
\end{figure}
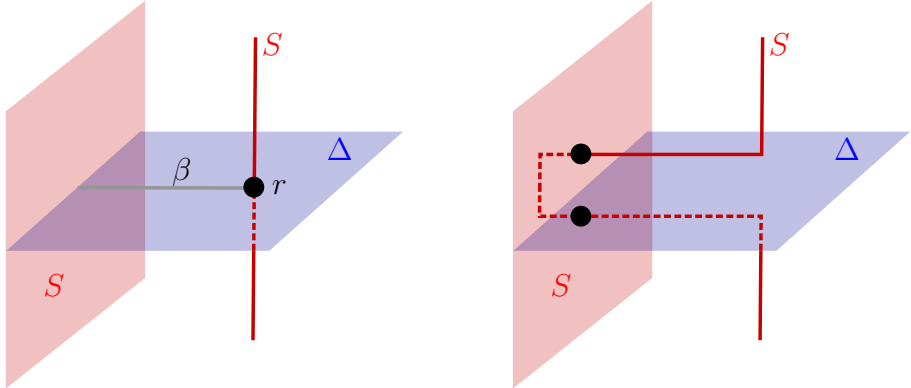 
 
 \end{proof}

For any link $K$ in any homology sphere $Y$ we can find a simply connected homology cobordism $W$ and an immersed union of annuli $A\subseteq W$ bounded by $K$ and the unlink in $S^3$.  Appealing to Lemma~\ref{lem: surface pi1 onto}, we arrange that $\pi_1(S^3\setminus \nu(U))\onto \pi_1(W\setminus \nu(A))$.  An iterative application of the following lemma allows us to extend $A$ to a symmetric relative  Whitney tower of arbitrary height.  We delay its proof until the end of the section.

\begin{lemma}[Compare to Lemma 5.4~\cite{DOP22}]\label{lem: height raising}
Let $W$ be a 4-manifold with boundary.  Let $T\subseteq W$ be a symmetric relative Whitney tower of height $h$.  Let $Y\subseteq \bdry W$ be a connected 3-manifold which contains all of the relative Whitney arcs of $T$.  Assume that every  component of $T_0$ has nonempty intersection with $Y$ and that the inclusion induced map $\pi_1(Y\setminus \nu(T))\onto \pi_1(W\setminus \nu(T))$ is onto.  

Then $T$ extends to a symmetric relative Whitney tower $T'$ of height $h+1$ such that $Y$ contains all of the relative Whitney arcs of $T'$ and such that the inclusion induced map $\pi_1(Y\setminus \nu(T'))\onto \pi_1(W\setminus \nu(T'))$ is onto.
\end{lemma}

Between Lemmas~\ref{lem: surface pi1 onto}, \ref{lem: height raising}, and \ref{lem: relative Whitney tower to Whitney tower} we can prove Theorem~\ref{thm:main}.  We recall its statement.

\begin{reptheorem}{thm:main}
Let $h\in \N$ and $K$ be a link in a homology sphere.  There is some link in $S^3$ which is height $h$ Whitney tower concordant to $K$.  
\end{reptheorem}

\begin{proof}
Let $K$ be a link in a homology sphere $X$.  By \cite[Corollary 9.3]{FQ} $X$ bounds a simply connected homology ball.  Removing an open 4-ball from the interior, we see a simply connected homology cobordism $W$ from $X$ to $S^3$.  

Let $A$ be an immersed union of annuli bounded $K$ and the unlink in $S^3$.  This union of annuli exists since $W$ is simply connected.  By Lemma~\ref{lem: surface pi1 onto} we  arrange that $\pi_1(S^3\setminus \nu(A))\onto \pi_1(W\setminus \nu(A))$ is surjective.  Since a height $0$ symmetric relative Whitney tower is just an immersed surface,  the assumptions of Lemma~\ref{lem: height raising} are satisfied by $T=A$, $Y=S^3$, and $h=0$.  The tower produced by Lemma~\ref{lem: height raising} satisfies all of the assumptions of Lemma~\ref{lem: height raising} and so we can appeal to the lemma again.  Iterating, we see that $A$ extends to a height $h-1$ symmetric relative Whitney tower in $W$ and that all of its relative Whitney arcs are contained in $S^3\subseteq W$.  Finally, by Lemma~\ref{lem: relative Whitney tower to Whitney tower} there is a height $h-1$ symmetric Whitney tower whose base surface is an annulus bounded by $K$ together with some link $J$ in $S^3$.  

  It remains only to arrange that there is a trivialization of the normal bundle of $A$ which restricts to the $0$ framings on $K$ and $J$.  We  do so one component at a time.  First extend the $0$ framing on $K_i$ over $A_i$ (the component of $A$ bounded by $K_i\cup-J_i$).  This can be done since an annulus deformation retracts to its either boundary component.  That this restricts to $J_i$ to give the $f_i$-framing for some $f_i\in \Z$.    Let $F_i\subseteq Y$ be a Seifert surface for $K_i$.  Cap the $Y$-boundary component of $W$ with a contractible 4-manifold. Call the result $\mathcal{B}$.  Since $\mathcal{B}$ is a contractible 4-manifold bounded by $S^3$, $\mathcal{B}$ is homeomorphic to the 4-ball by the topological 4-dimensional Poincar\'e conjecture.     Since the $0$-framing of $K_i$ extends over both $A_i$ and $F_i$,  $S_i = A_i\cap F_i$ is a framed surface in the 4-ball which restricts to the $f_i$-framing of $-J_i$.  Then a framed pushoff $S_i^+$ is bounded by $(-J_i)^+$, the $f_i$-framed pushoff of $-J_i$.  
  
\begin{figure}[h]
         \centering
         \begin{tikzpicture}
        \node at (0,.2){\includegraphics[width=.25\textwidth]{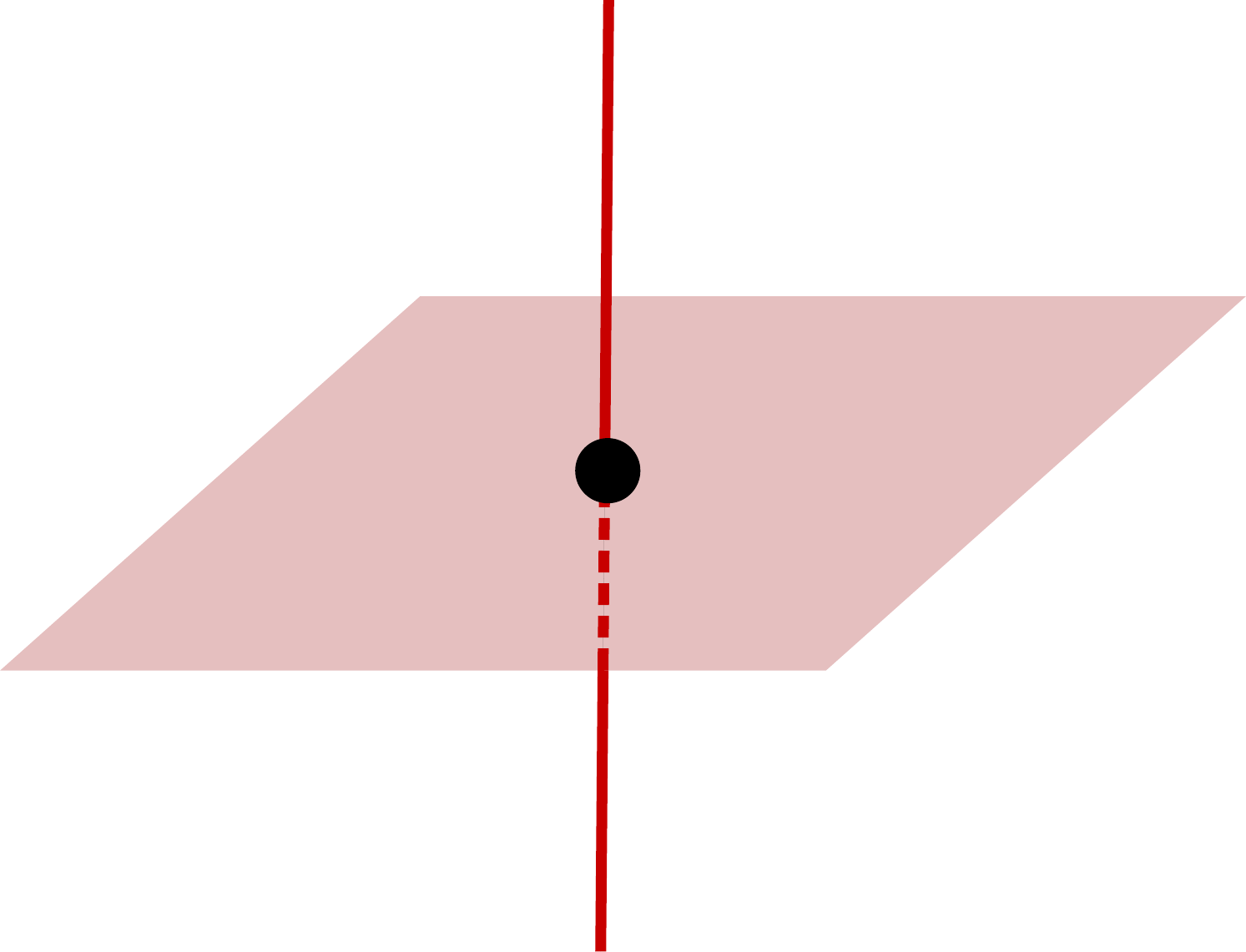}};
        \node[] at (1,.3) {$S$};
        \node[] at (-.3,-1) {$S$};
        \node at (6,.2){\includegraphics[width=.25\textwidth]{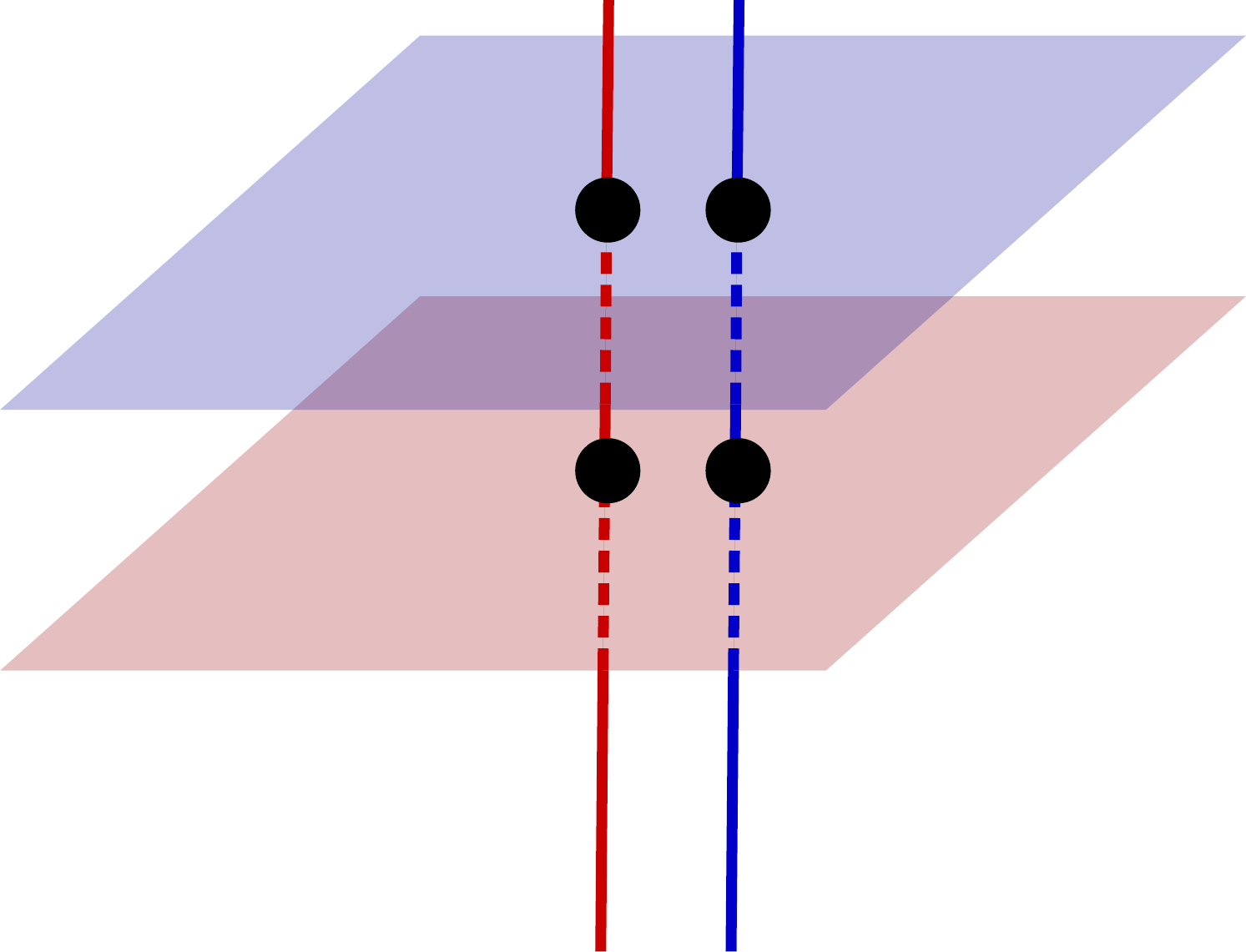}};
        
        \node[] at (7,.3) {$S$};
        \node[] at (7,1.3) {$S^+$};
        \node[] at (5.7,-1) {$S$};
        \node[] at (6.65,-1) {$S^+$};
         \end{tikzpicture}
        \caption{Left: a double point in a framed surface $S$.  Right: A pushoff $S^+$ intersects the original surface in two points with the same sign as that double point.
         }
        \label{fig:doublePointPushoff}
\end{figure}
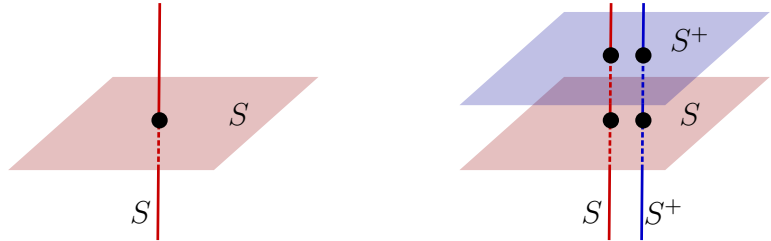 
  
Each double point of $S_i$ results in two points of intersection with the same sign between $S_i$ and $S_i^+$, as in Figure~\ref{fig:doublePointPushoff}.  We may as well assume that $h-1\ge1$, so that all of the self intersections of $S_i$ are paired by Whitney disks, and in particular the signed count of the number of intersections between $S_i$ and $S_i^+$ is zero.     Finally since linking numbers can be computed as the signed count of the number of intersections between immersed surfaces in the 4-ball, 
$$f_i=\lk(-J_i,(-J_i)^+) = S_i\cdot S_i^+ = 0.$$
Therefore, the $0$-framings on $K$ and $J$ each extend over $A$, and $A$ extends to a height $h-1$ Whitney tower.  Thus, $K\sim_h J$ completing the proof

%
%
\end{proof}

%
%

Finally, we prove Lemma~\ref{lem: height raising} by appealing to Lemma~\ref{lem: surface pi1 onto}.  

\begin{proof}[Proof of Lemma~\ref{lem: height raising}]
Let $W$, $T$, and $Y$ be as in the statement.  
Recall that $T_h$ is the height $h$ part of $T$.  Let $p^1,p^2,\dots,p^k$  be the double points in $T_h$.  Let $\psi:S\to W$ is the immersion of a surface parametrizing $T_h$.  For any $j=1,\dots, k$, let $\psi^{-1}\{p^j\}=\{p_1^j,p_2^j\}$.   There are embedded arcs $\alpha_1^j$ and $\alpha_2^j$ in $S$ running from $p_1^j$ and $p_2^j$ to points $q_1^j$ and $q_2^j$ interior to $\psi^{-1}(Y)$.  We will further arrange that $\alpha_1^1, \alpha_2^1,\dots, \alpha_1^k,\alpha_2^k$ are all disjoint from each other.  
  Such arcs exist when $h=0$ because $Y$ has nonempty intersection with each component of $T_0$ and when $h>0$ because $Y$ contains every relative Whitney arc.  Since $\pi_1(Y\setminus \nu(T))\onto \pi_1(W\setminus \nu(T))$, for $j=1,\dots, k$, there is some arc $\alpha_3^j\subseteq Y$ running between $\psi(q_1^j)$ and $\psi(q_2^j)$
   so that $\psi\left(\alpha_1^j\right)*\alpha_3^j*\psi\left(\overline{\alpha_2^j}\right)$ bounds an immersed disk $\Delta^j$ with interior disjoint from $T$.  Notice that $T\cup\left(\Delta^1\cup\dots\cup\Delta^k\right)$ is now a height $h+1$ relative Whitney tower.  
   
     The assumptions of Lemma~\ref{lem: surface pi1 onto} are satisfied by $W'=W\setminus \nu(T)$,  $Y'=Y\setminus \nu(T)\subseteq \bdry W'$ and the immersed surface $S'=\Delta^1\cup\dots, \Delta^k\subseteq W'$.  Thus, we can change $\Delta^1\cup\dots\cup \Delta^k$ by a sequence of finger moves in $W\setminus \nu(T)$ so that $\pi_1(Y\setminus \nu(T\cup \Delta^1\cup\dots\cup \Delta^k)) \onto \pi_1(W\setminus \nu(T\cup \Delta^1\cup\dots\cup \Delta^k))$ is surjective, completing the proof.  

\end{proof}


\section{Whitney tower concordance and solvable concordance}\label{sect: tower implies sol}

For any link $K\subseteq S^3$,  If $K\in \mathcal{W}_{h+2}$ then $K$ is $h$-solvable \cite[Theorem 8.12]{COT03}.  In this section we show that if two links in homology spheres are height $h+2$-Whitney tower concordant then they are $h$-solvably concordant.  We begin by recalling the definition of $h$-solvable concordance.

\begin{definition}[Definition 2.3 of \cite{Davis2019}]\label{defn: solvable concordance}
Let $K\subseteq X$ and $J\subseteq Y$ be links in  homology spheres.  We say that $K$ is \emph{$n$-solvably concordant} to $J$ if there exists a cobordism $W$ from $X$ to $Y$ in which the components of $K$ and $J$ cobound a disjoint collection of embedded locally flat annuli $A\subseteq W$ such that:
\begin{enumerate}
\item $H_1(W)=0$, so that $W$ is an $H_1$-cobordism. 
\item \label{solvable surfaces} There exist locally flat embedded closed oriented surfaces in $W$ with trivial normal bundles $L_1,D_1,\dots, L_k,D_k$ all disjoint from $C$ and from each other except that for each $i=1,\dots, k$, the surfaces $L_i$ and $D_i$ intersect transversely in a single point.
   \item \label{solvable basis for H2} $\{[L_1],[D_1],\dots,[L_k],[D_k]\}$ generates $H_2(W)$.
\item \label{solvable lift} For all $i=1,\dots, k$, the images of $\pi_1(L_i)\to \pi_1(E(A))$ and $\pi_1(D_i)\to \pi_1(E(A))$ are both contained in $\pi_1(E(A))^{(n)}$.  
\end{enumerate}
\end{definition}

Here $E(A)=W\setminus\nu(A)$ is the exterior of $A$ and  $\pi_1(E(A))^{(n)}$ refers to the \emph{derived series} of the group $\pi_1(E(A))$.  It is defined recursively as  $\pi_1(E(A))^{(0)} = \pi_1(E(A))$ and $\pi_1(E(A))^{(n+1)} = [\pi_1(E(A))^{(n)},\pi_1(E(A))^{(n)}]$.

\begin{remark}
Similar to Remark~\ref{rmk: smoothing}, one gets an equivalent definition by requiring these surfaces to be smooth, either with regard to a smooth structure on $W\setminus \nu(A)$, if one exists, or a smooth structure on $W\setminus(\nu(A)\cup\{p\})$.  Indeed by \cite[Theorem 8.1A]{FQ} $L_1\cup D_1\cup\dots \cup L_k\cup D_k$ admits a regular homotopy inside of a small neighborhood to a smoothly immersed surface.  Since a regular homotopy preserves  intersection numbers, $L_i$ still has algebraically zero self intersections, and so by increasing the genus as in Figure~\ref{fig: tubing} we can arrange that the geometric intersections between $L_1\cup D_1\cup\dots\cup L_k\cup D_k$ are preserved.   As this occurs in a small neighborhood of $L_1\cup D_1\cup\dots\cup L_k\cup D_k$, the image of $\pi_1(L_i)\to \pi_1(E(A))$ and $\pi_1(D_i)\to \pi_1(E(A))$ is still contained in $\pi_1(E(A))^{(n)}$.  
\end{remark}

\begin{figure}[h]
         \centering
         \begin{tikzpicture}
        \node at (0,.2){\includegraphics[width=.1\textwidth]{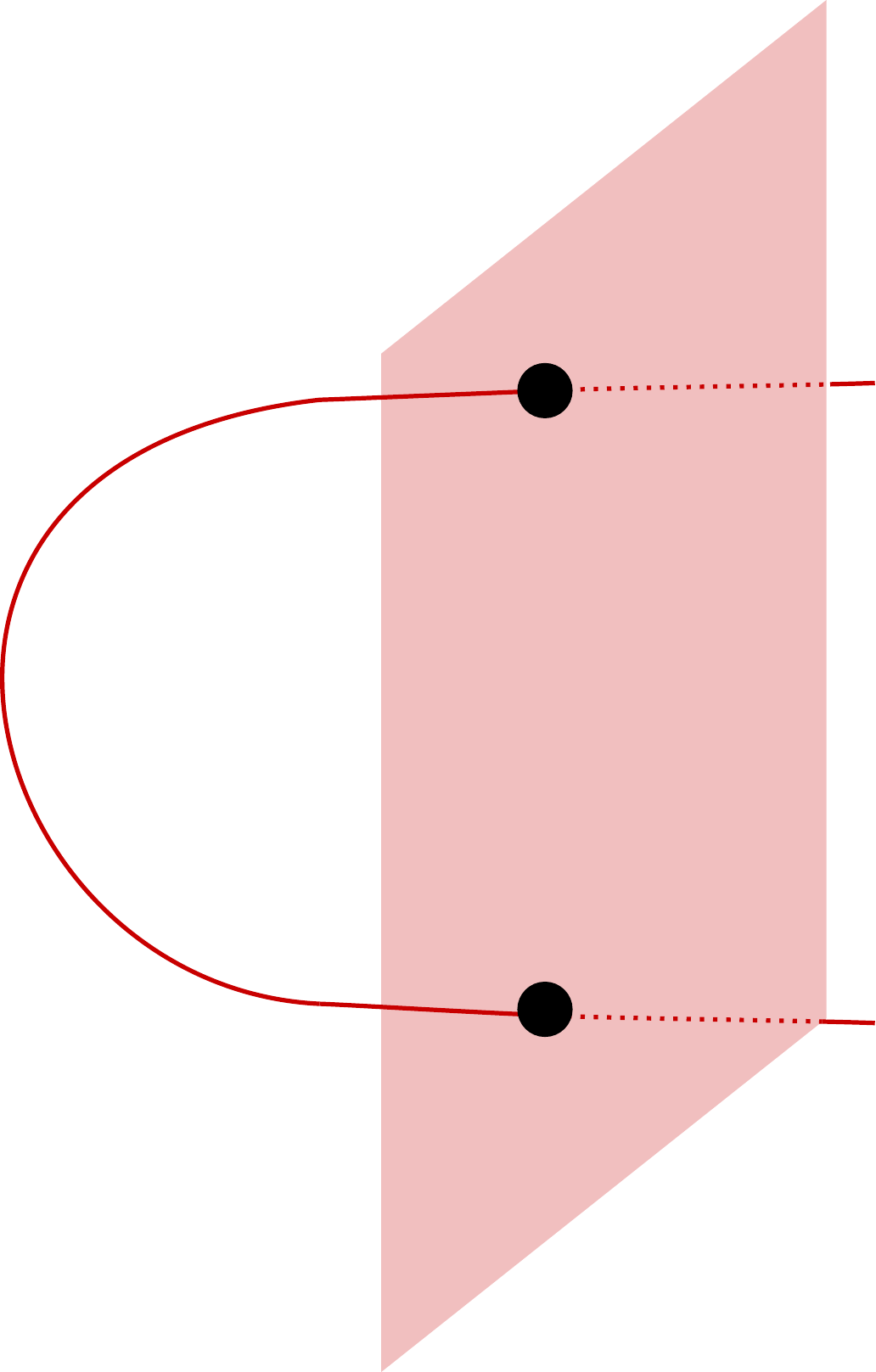}};
        \node at (6,.2){\includegraphics[width=.1\textwidth]{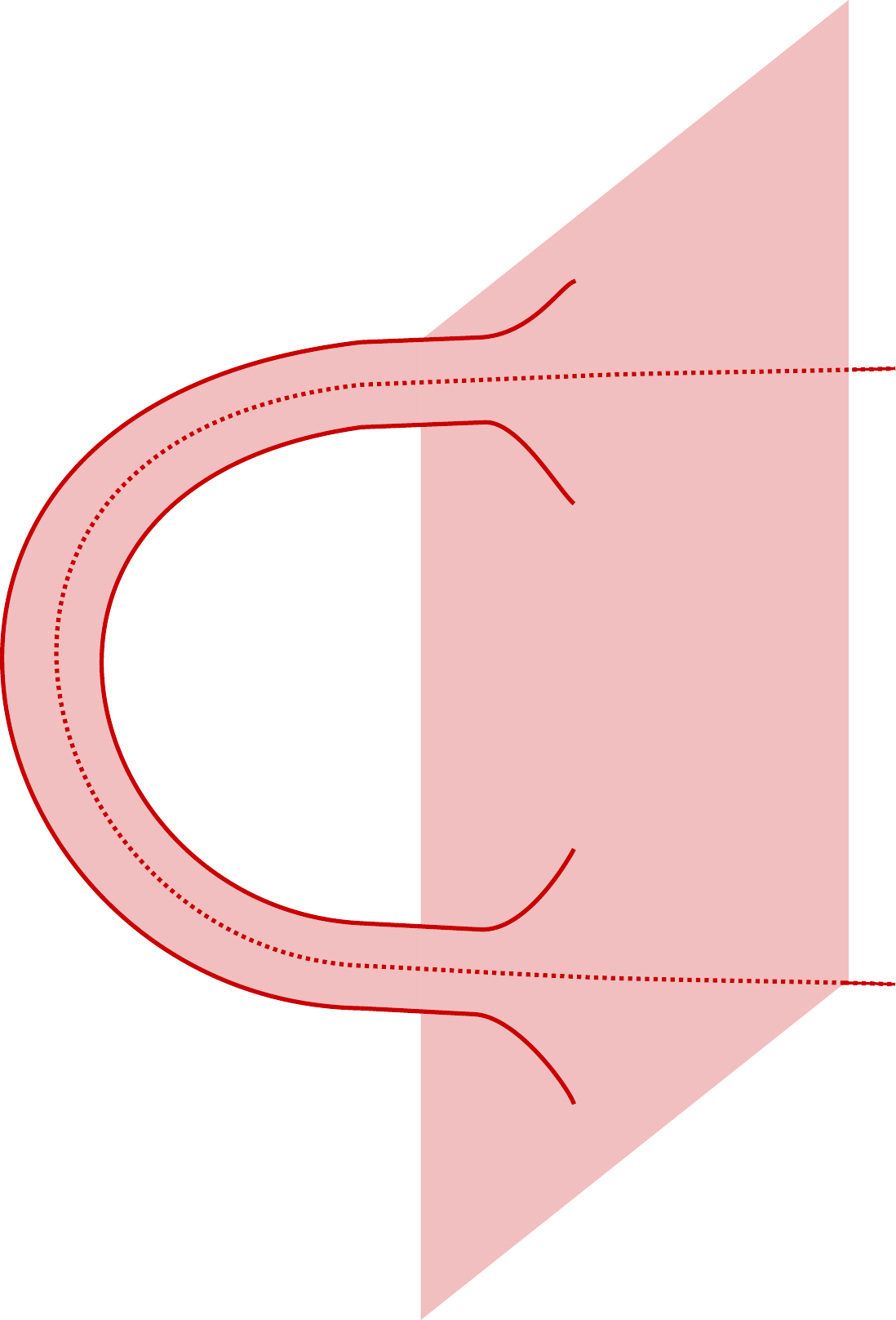}};
         \end{tikzpicture}
        \caption{A pair of double points of opposite sign in an immersed surface may be cancelled at the cost of increasing the genus of the surface by 1.
         }
        \label{fig: tubing}
\end{figure} 

There is a geometric interpretation of the derived series of the fundamental group in the language of gropes.  A \emph{height $1$ symmetric grope} is a compact oriented  surface.  A height $h+1$ symmetric grope consists of a surface together with a symplectic basis $\{a_1,b_1, \dots, a_g, b_g\}$ and $2g$ many height $h$ gropes bounded by $a_1, b_1, \dots, a_g, b_g$.  Generally in the literature, these are taken to be embedded in a 4-manifold and have some framing condition, but these conditions are not relevant to our work.    For any 4-manifold $X$, a straighforward inductive argument reveals that $\gamma\in\pi_1(X)^{(h)}$ if and only if $\gamma$ bounds an immersed height $h$ grope in $X$.  

\begin{proposition}\label{prop: tower to solution}
Let $K$ and $J$ be links in homology spheres.  If $K$ and $J$ are height $h+2$ Whitney tower concordant, then $K$ and $J$ are $h$-solvably concordant.  
\end{proposition}
\begin{proof}
Let $K$ and $J$ be links in homology spheres $X$ and $Y$, respectively.  Let $W$ be a homology cobordism from $X$ to $Y$ in which the components of $K$ and $J$ cobound an immersed union of  annuli $A$ which extends to a height $h+1$ symmetric Whitney tower $T$.  

Let $\Delta_1,\dots, \Delta_k$ be a complete list of all of the height $h+1$ Whitney disks in $T_{h+1}$.  For every $j=1,\dots, k$, let $\gamma_j$ be the framed curve given by pushing $\bdry \Delta_j$ slightly into the interior of $\Delta_j$.  Let $W'$ be the result of modifying $W$ by surgery along $\gamma_1,\dots, \gamma_k$.

First we exhibit a disjoint union of embedded annuli in $W'$ cobounded by the components of $K$ and the components of $J$.  Let $T'\subseteq W'$ be the height $h+1$ Whitney tower given starting with $T$ and replacing each $\Delta_j\subseteq T_{h+1}$ with the embedded disk glued in along $\gamma_j$.  Notice that all of the height $h+1$ Whitney disks in $T'$ are embedded and have interiors disjoint from all other surfaces in $T'$.  Thus, we may use the Whitney trick to eliminate all of the intersections in $T'_{h}$.  Next use these embedded height $h$ Whitney disks to eliminate intersections in $T'_{h-1}$.  Iterate until $T'_0$ is embedded.  We now have that in $W'$ the components of $L$ and $J$ cobound an embedded union of disjoint annuli, which we will  call $A$.  Observe that $A$ is contained in a small neighborhood of $T'$.    The remainder of the proof amounts to showing that the conditions of Definition~\ref{defn: solvable concordance} are satisfied.

As modifying a 4-manifold by surgery along nullhomologous simple closed curves does not change first homology, $H_1(W')=0$.  Therefore $W'$ is an $H_1$-cobordism.  Each time we perform surgery we increase the rank of second homology by $2$.  Let $L_i^0$ be a meridonal 2-sphere for $\gamma_i$ (an embedded sphere bounding  a 3-ball in $W$ which intersects $\gamma_i$ transversely in a single point).  Let $D_i^0\subseteq W'$ be the immersed sphere given by capping the disk in $\Delta_i$ bounded by $\gamma_i$ with the disk glued to the surgery curve $\gamma$.  We now have that $\{[L_1^0],[D_1^0],\dots,[L_k^0],[D_k^0]\}$ forms a basis for $H_2(W')$.  Being the boundary of an embedded 3-sphere in $W$, $L_i^0$ has trivial normal bundle.  Since the faming on $\gamma_i$ used to perform surgery was the framing induced by $\Delta_i$, $D_i^0$ has trivial normal bundle.

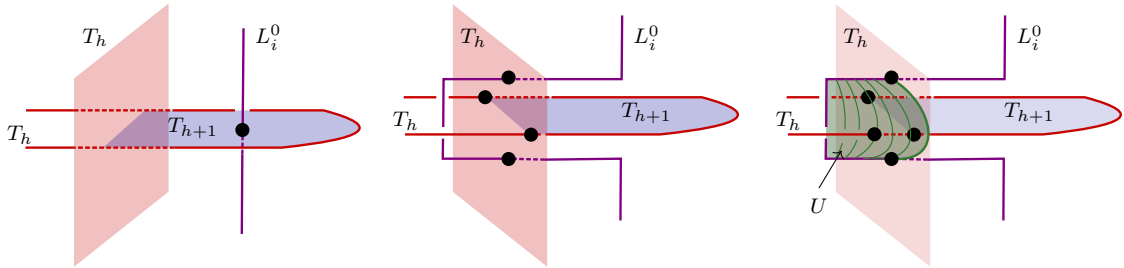
\begin{figure}[h]

     \begin{subfigure}[b]{0.3\textwidth}
         \centering
         \begin{tikzpicture}
        \node at (0,.2){\includegraphics[width=.9\textwidth]{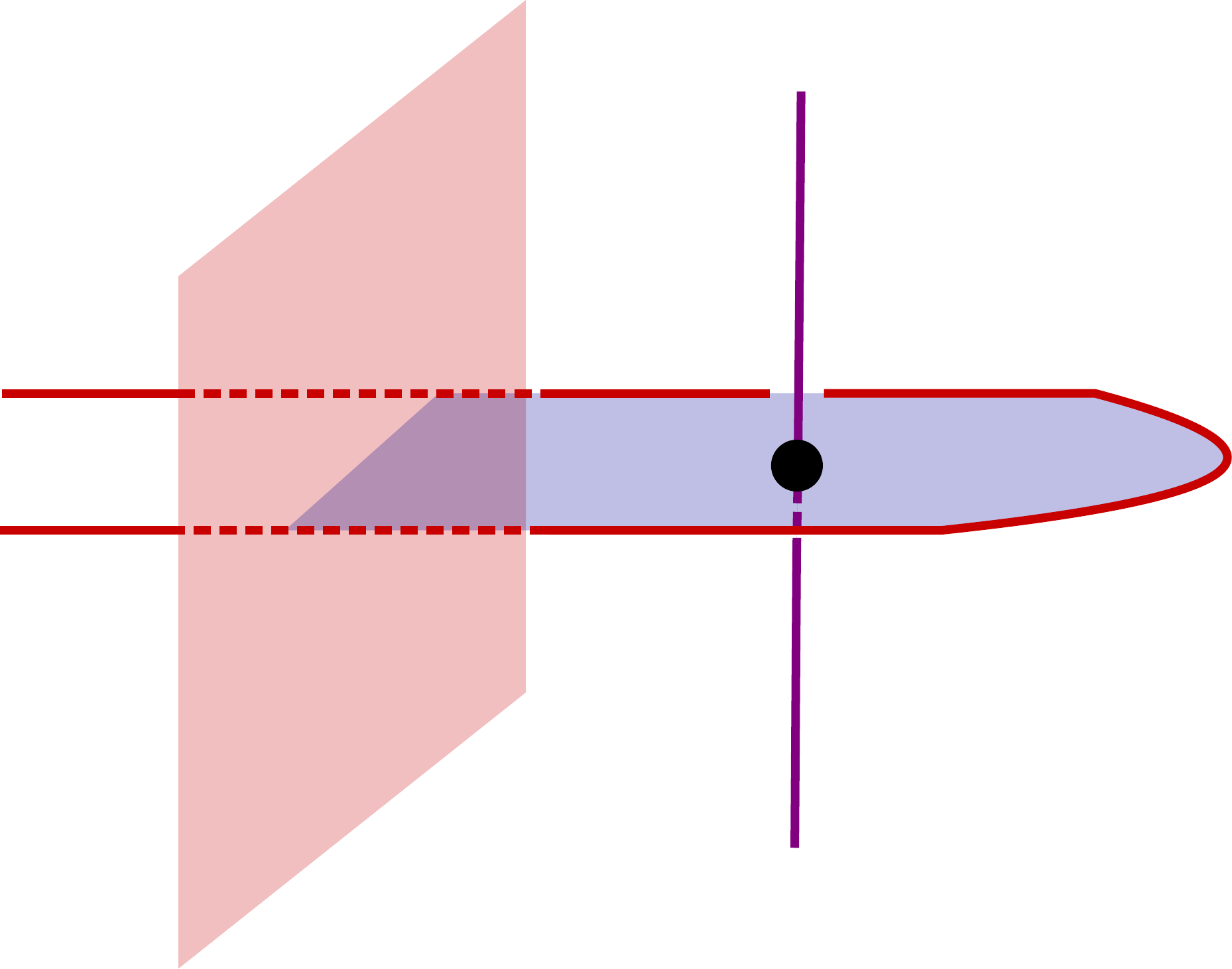}};
        \node at (0,.3){\tiny $T_{h+1}$};
        \node at (1,1.5){\tiny $L_i^0$};
        \node at (-1.3,1.5){\tiny $T_{h}$};
        \node at (-2.3,.2){\tiny $T_{h}$};
        \end{tikzpicture}
        \end{subfigure}
     \begin{subfigure}[b]{0.3\textwidth}
         \centering
         \begin{tikzpicture}
        \node at (0,.2){\includegraphics[width=.9\textwidth]{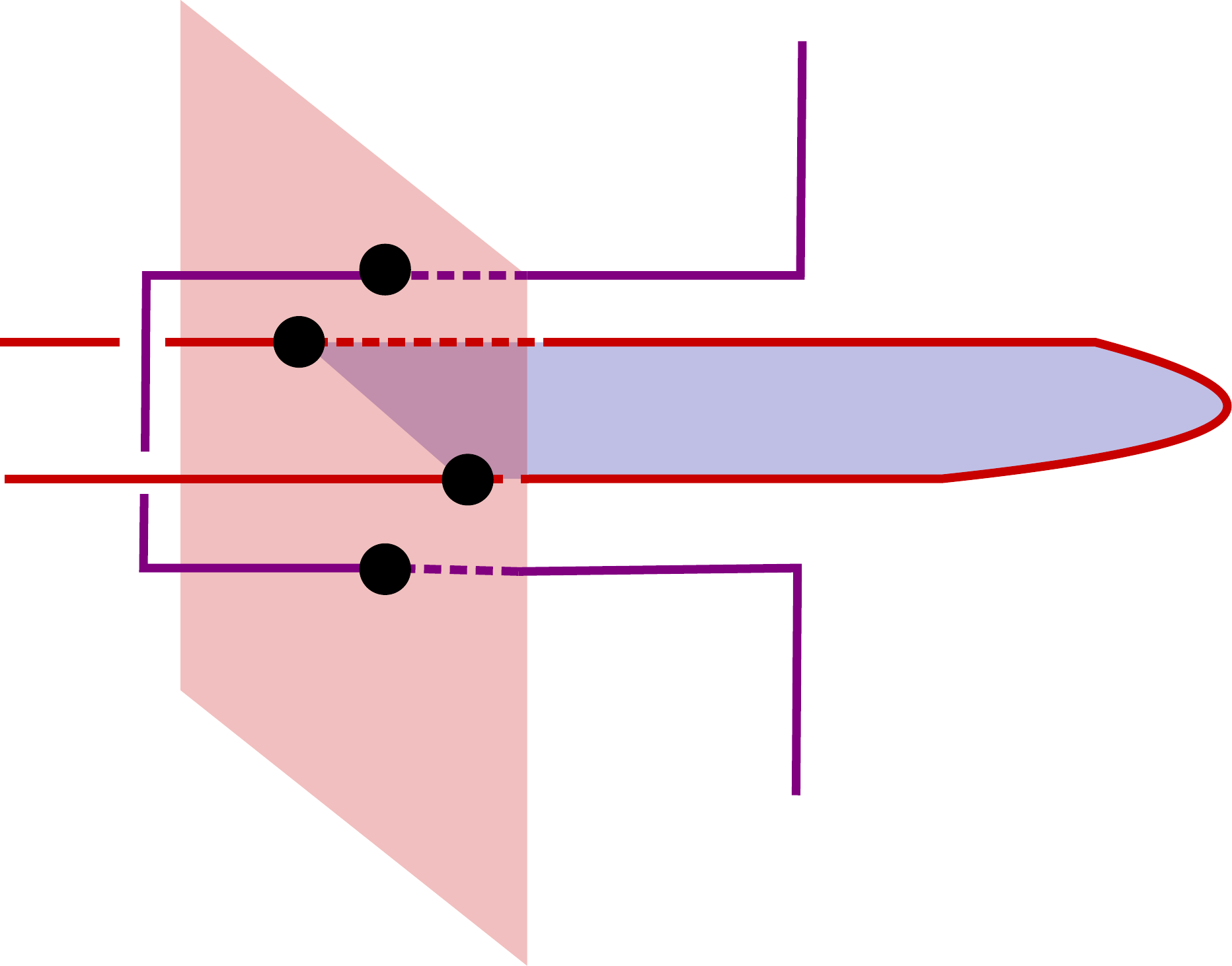}};
        \node at (1,.5){\tiny $T_{h+1}$};
        \node at (1,1.5){\tiny $L_i^0$};
        \node at (-1.3,1.5){\tiny $T_{h}$};
        \node at (-2.2,.4){\tiny $T_{h}$};
        \end{tikzpicture}
        \end{subfigure}
     \begin{subfigure}[b]{0.3\textwidth}
         \centering
         \begin{tikzpicture}
        \node at (0,.2){\includegraphics[width=.9\textwidth]{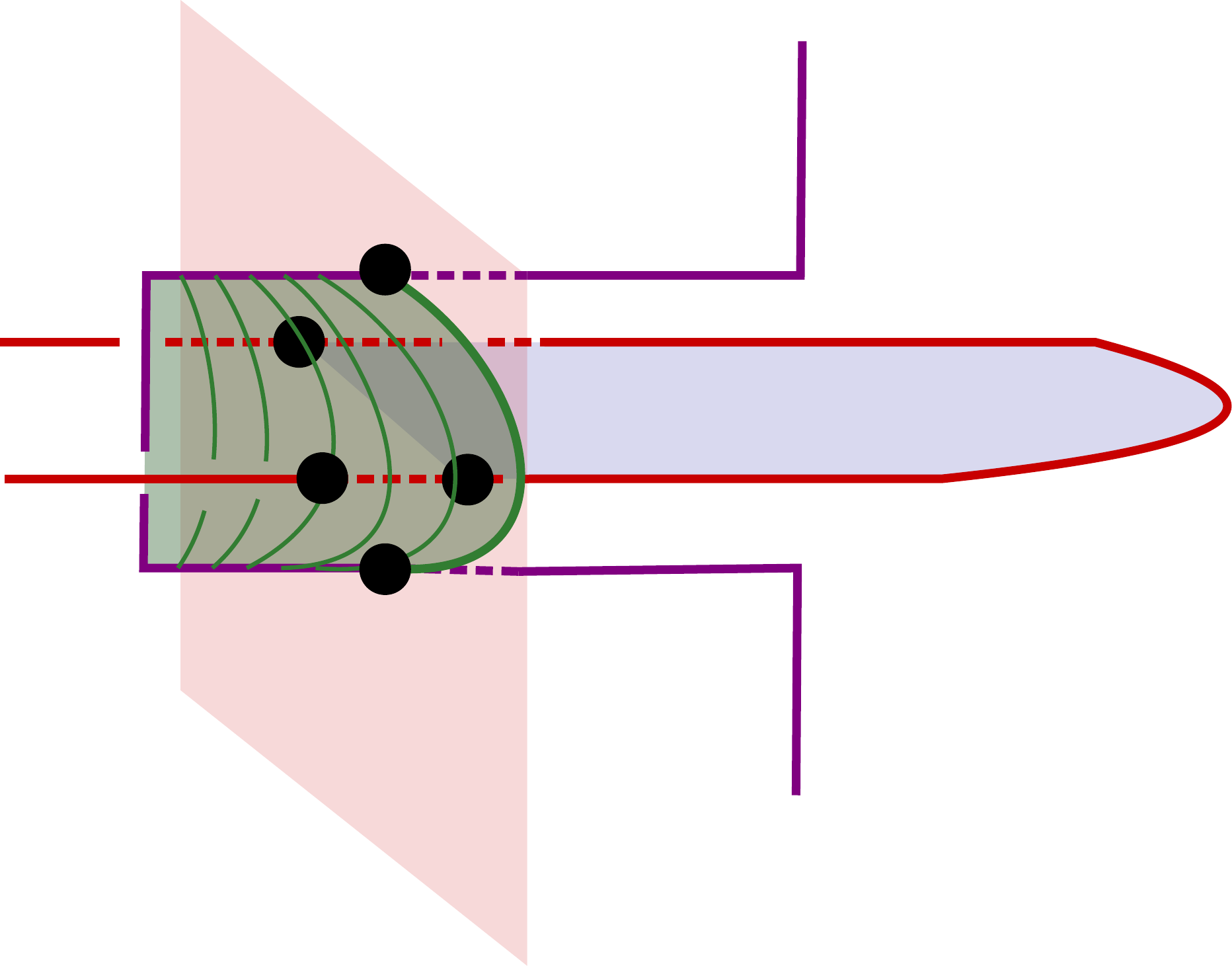}};
         \node at (1,.5){\tiny $T_{h+1}$};
        \node at (1,1.5){\tiny $L_i^0$};
        \node at (-1.3,1.5){\tiny $T_{h}$};
        \node at (-2.2,.4){\tiny $T_{h}$};
        \node[below] at (-1.8,-.5){\tiny $U$};
        \draw[->](-1.8, -.5)--(-1.5,0);
         \end{tikzpicture}
         \end{subfigure}
        \caption{Left: $L_i^0$ intersects $T_{h+1}$ transversely in a single point.  Center: After a finger move $L_i^0$ (purple) intersects $T_h$ (red) in two points with opposite sign.  Right: An embedded Whitney disk $U$ pairing these two points of intersection and which intersects $T_h$ transversely in one point.  
         }
        \label{fig: Fix Lagrangian 1}
\end{figure}

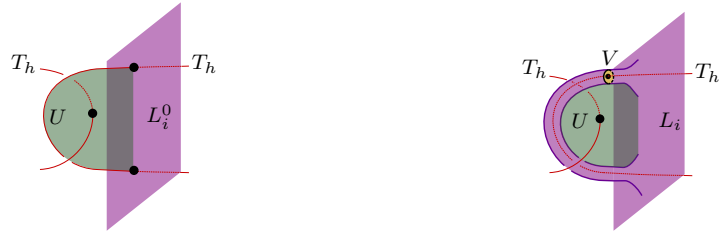
\begin{figure}[h]

     \begin{subfigure}[b]{0.3\textwidth}
         \centering
         \begin{tikzpicture}
        \node at (0,0){\includegraphics[width=.4\textwidth]{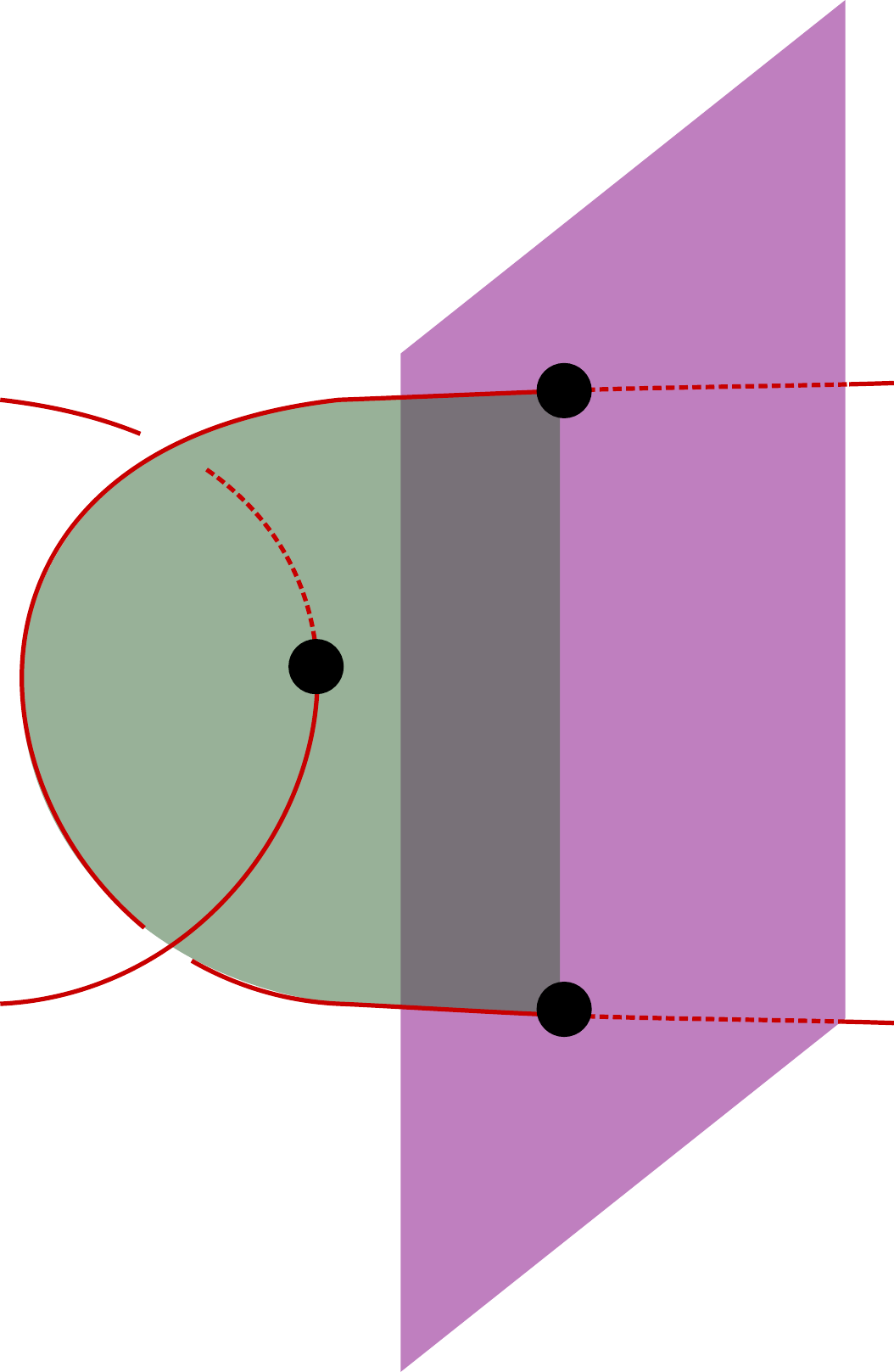}};
        \node at (-.75,0){\tiny$U$};
        \node at (-1.2,.7){\tiny$T_h$};
        \node at (1.2,.7){\tiny$T_h$};
        \node at (.6,0){\tiny$L_i^0$};
        \end{tikzpicture}
        \end{subfigure}\hspace{.1\textwidth}
     \begin{subfigure}[b]{0.3\textwidth}
         \centering
         \begin{tikzpicture}
        \node at (0,0){\includegraphics[width=.4\textwidth]{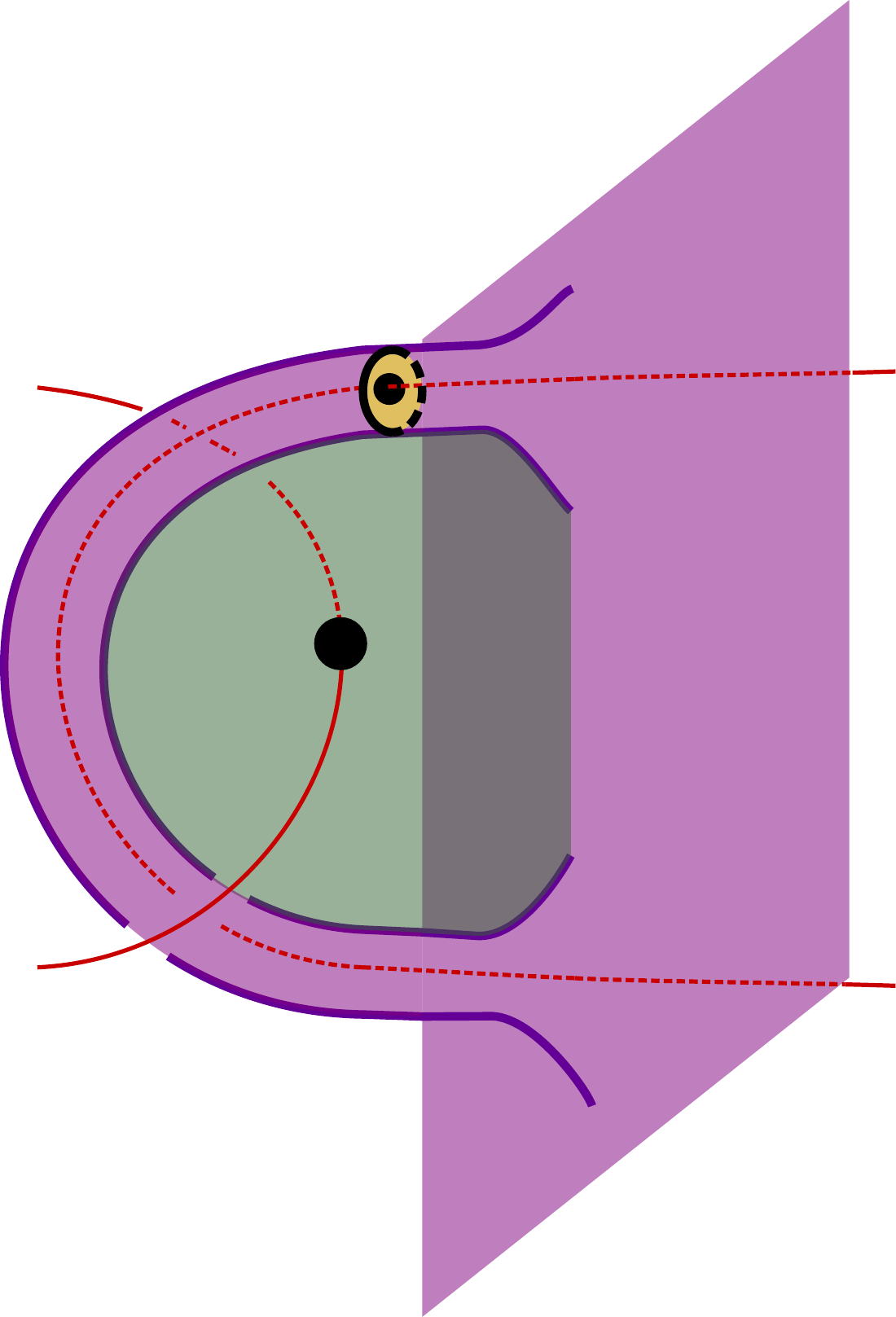}};
        \node at (-.1,.85){\tiny$V$};
        \node at (-.5,0){\tiny$U$};
        \node at (.7,0){\tiny$ L_i$};
        \node at (-1.1,.7){\tiny$T_h$};
        \node at (1.2,.65){\tiny$T_h$};
         \end{tikzpicture}
         \end{subfigure}
        \caption{Left: Two points of intersection between $L_i^0$ and $T_h$ which are paired by an embedded Whitney disk $U$ intersecting $T_h$ in one point Right: Tubing produces a surface of genus 1 with a symplectic basis bounding embedded disks $U$ and $V$ each of which intersects $T_h$ in a single point.  
         }
        \label{fig: Fix Lagrangian 2}
\end{figure} 

Our next goal is to replace $L_i^0$ with an embedded surface homologous to $L_i^0$ which extends to a height $h+1$ grope in the exterior of $T'$.  Currently, $L_i^0$ intersects $T_{h+1}'$ transversely in a single point.  By performing a finger move we isotope $L_i^0$ to be disjoint from $T_{h+1}'$ and to instead intersect $T'_{h}$ in two points of opposite sign.  As in Figure~\ref{fig: Fix Lagrangian 1}, these two points of intersection are paired with an embedded Whitney disk $U$ disjoint from $T_{h+1}$ whose interior intersects $T_h$ transversely in a single point.  As in Figure~\ref{fig: Fix Lagrangian 2} we tube $L_i^0$ to itself along the arc $\bdry U\cap T_h$ to get an embedded genus 1 surface $L_i$ disjoint from $T$.  
  Also in Figure~\ref{fig: Fix Lagrangian 2} we see that a pair of disks $U$ and $V$ bounded by a symplectic basis for $L_i$.  These disks each intersect $T_h$ transversely in a single point and  $L_i\cup U\cup V$ is a height $2$ grope.

Next we apply exactly the same steps as the preceding paragraph to $U$ and $V$ this time instead of $L^0$. First, modify $U$ and $V$ by a finger move so that they are disjoint from $T_{h}$ and each intersect $T_{h-1}$ in two points paired with a Whitney disk.  Then add a tube to each to replace $U$ and $V$ by genus 1 surfaces which admit symplectic bases bounding disks that each intersect  $T_{h-1}$ in a single point. The genus 1 surface $L_i$ now  forms the height $1$ part of a height $3$ grope $G$ which is disjoint from $T'$ except that the height 3 part of $G$ consists of disks each intersecting the $T'_{h-1}$ in a single point.  

  For any $k\in \{1,\dots,h+2\}$ we iterate this construction to produce a height $k$ grope $G_i$ with height 1-surface $L_i$, whose height $k$ part consists of disjoint embedded disks each intersecting $T'_{h+2-k}$ transversely in a single point, and which is otherwise disjoint from $T'$.  Setting $k=h+2$ and dropping the disks that make up the height $h+2$ part of $G$, we see that $L_i$ forms the height $1$ part of a height $h+1$ grope in the exterior of $T'$.  It follows that this grope is disjoint from $A$ since $A$ is contained in a neighborhood of $T'$.  As the sphere $L_i^0$ has trivial normal bundle so does $L_i$.  As $L_i$ forms the height $1$ part of a height $h+1$ grope disjoint from $A$, there is a generating set for $\pi_1(L_i)$ consisting of curves that bound height $h$ gropes disjoint form $A$.  We conclude that the image of $\pi_1(L_i)\to \pi_1(E(A))$ is contained in $\pi_1(E(A))^{(h)}$.  

Next consider $D_i^0$, the immersed sphere formed from the union of $\Delta_i$ with the disk attached to $\gamma_i$.  As the framing on $\gamma_i$ used to perform surgery is the same as that coming from the normal bundle of $\Delta_i$, $D_i^0$ has trivial normal bundle.  These spheres are not embedded, nor are they disjoint from each other, but they are disjoint from $T'$ and they are disjoint from $L_1\cup\dots\cup L_k$ except that $D_i^0$ intersects $L_i$ transversely in a single point.  

For each $i=1,\dots, k$ and each point of self intersection of $D_i^0$, tube $D_i^0$ to a pushed off copy of $L_i$.  This increases the genus of $D_i^0$ by 1, and reduces the number of self intersections by 1.  The homology class $[D_i^0]$ is replaced by $[D_i^0]\pm[L_i]$ depending on the sign of the intersection point removed.  In particular $\{[L_1], [D_1^0], \dots ,[L_k], [D_k^0]\}$ still generates $H_2(W')$.  Since the result of tubing together framed surfaces is still framed, $D_i^0$ is still framed.  By making this replacement at each point of self intersection, we arrange that $D_i^0$ is embedded.    Similarly, we  eliminate all points of intersection between any $D_i^0$ and $D_j^0$ by tubing $D_i^0$ into a pushed off copy of $L_j$.  Let $D_1,\dots, D_k$ be the resulting surfaces.  Each $D_i$ consists of a union of pushoffs of $L_1,\dots, L_k$ together with a planar surface.    As a consequence, the image of $\pi_1(D_i)\to \pi_1(E(A))$ is contained in the normal closure of the subgroup generated by the images of $\pi_1(L_1), \dots \pi_1(L_k)$.  As we have already explained, this is contained in $\pi_1(E(A))^{(h)}$.  

Thus, $L_1,D_1,\dots, L_k, D_k$ satisfy conditions~(\ref{solvable surfaces}), (\ref{solvable basis for H2}), and (\ref{solvable lift}) of Definition~\ref{defn: solvable concordance}.  Therefore $K$ and $J$ are $h$-solvably concordant, completing the proof.  
\end{proof}


\bibliographystyle{alpha}

\bibliography{biblio}

\end{document}